\documentclass[12pt]{amsart}

\usepackage{amsmath}
\usepackage{amsfonts}
\usepackage{amssymb}
\usepackage[all]{xy}           
\usepackage{bbding}
\usepackage{txfonts}
\usepackage{amscd}

\usepackage[shortlabels]{enumitem}
\usepackage{ifpdf}
\ifpdf
  \usepackage[colorlinks,final,
  hyperindex]{hyperref}
\else
  \usepackage[colorlinks,final,
  hyperindex]{hyperref}
\fi
\usepackage{tikz}
\usepackage[active]{srcltx}

\makeatletter
\numberwithin{equation}{section}

\newtheorem{thm}{Theorem}[section]
\newtheorem{lem}[thm]{Lemma}
\newtheorem{cor}[thm]{Corollary}
\newtheorem{pro}[thm]{Proposition}
\newtheorem{ex}[thm]{Example}
\newtheorem{rmk}[thm]{Remark}
\newtheorem{defi}[thm]{Definition}

\newcommand{\fu}{\mathbf u}
\newcommand{\fl}{\mathbf l}
\newcommand{\fr}{\mathbf r}
\newcommand{\ffl}{\tilde{\mathbf{l}}}
\newcommand{\ffr}{\tilde{\mathbf{r}}}

\newcommand{\g}{\mathfrak{g}}

\newcommand{\gl}{\mathfrak{gl}}

\newcommand{\kl}{\mathfrak{l}}
\newcommand{\kr}{\mathfrak{r}}
\newcommand{\kkl}{\tilde{\mathfrak{l}}}
\newcommand{\kkr}{\tilde{\mathfrak{r}}}

\newcommand{\bz}{\mathbb{Z}}

\newcommand{\ad}{\mathrm{ad}}

\newcommand{\CYBE}{\mathrm{CYBE}}
\newcommand{\LYBE}{\mathrm{LYBE}}
\newcommand{\PYBE}{\mathrm{PYBE}}
\newcommand{\PoiYBE}{\mathrm{PoiYBE}}
\newcommand{\DPYBE}{\mathrm{DPYBE}}

\newcommand{\id}{\mathrm{id}}

\begin{document}

\title[Dual pre-Poisson bialgebras and Poisson bialgebras ]
{Quasi-triangular dual pre-Poisson bialgebras and its connection with Poisson bialgebras}

\author{Bo Hou}
\address{School of Mathematics and Statistics, Henan University, Kaifeng 475004, China}
\email{bohou1981@163.com}

\author{Ru Li}
\address{School of Mathematics and Statistics, Henan University, Kaifeng 475004,
China}
\email{13037698973@163.com}


\begin{abstract}
In this paper, the notions of quasi-triangular and factorizable dual pre-Poisson bialgebras
are introduced. A factorizable dual pre-Poisson bialgebra induces a factorization of the
underlying dual pre-Poisson algebra, and the double of any dual pre-Poisson bialgebra
is factorizable. We introduce the notion of quadratic Rota-Baxter dual pre-Poisson algebras
and show that there is a one-to-one correspondence between factorizable dual pre-Poisson
bialgebras and quadratic Rota-Baxter Poisson algebras of nonzero weights.
Moreover, a method of constructing infinite-dimensional dual pre-Poisson bialgebras
using finite-dimensional Poisson bialgebras is given. We prove that there is a
completed dual pre-Poisson bialgebra structure the tensor product of a Poisson bialgebra
and a quadratic $\bz$-graded perm algebra, and this completed dual pre-Poisson bialgebra
structure is coboundary (resp. quasi-triangular, triangular) if the original Poisson
bialgebra is coboundary (resp. quasi-triangular, triangular). The induced factorizable
finite-dimensional dual pre-Poisson bialgebras are considered.
\end{abstract}

\makeatletter
\@namedef{subjclassname@2020}{\textup{2020} Mathematics Subject Classification}
\makeatother

\maketitle

\vspace{-6mm}
\tableofcontents 



\vspace{-8mm}

\section{Introduction}\label{sec:intr}

In this paper, we further investigate the bialgebra theory of dual pre-Poisson algebras.
We establish the quasi-triangular and factorizable theories in the context of
dual pre-Poisson bialgebras, synthesizing concepts from both quasi-triangular perm
bialgebras and quasitriangular Leibniz bialgebras. Moreover, we also propose a method
for constructing infinite-dimensional dual pre-Poisson bialgebras using the Poisson bialgebras.

A Poisson algebra is both a Lie algebra and a commutative associative algebra which are
compatible in a certain sense, which serve as fundamental structures in various areas of
mathematics and mathematical physics, including Poisson geometry,
algebraic geometry, classical and quantum mechanics,quantum groups and quantization
theory \cite{Vai,Wei,GK,Pol,Arn,Hue,Kon}.
The notion of pre-Poisson algebras was first introduced by Aguiar \cite{Agu},
which combines a Zinbiel algebra and a pre-Lie algebra on the same vector space
satisfying some compatibility conditions. It is also known that a Rota-Baxter operator
on a Poisson algebra can produce a pre-Poisson algebra. The operad of pre-Poisson
algebras is Koszul \cite{DT}, and is the disuccessor of the operad of Poisson algebras
\cite{BBGN}.

The notion of dual pre-Poisson algebras (DPP algebras) was first given by Aguiar in
\cite{Agu}, which was also introduced in \cite{Uch} as the Koszul duality of the
pre-Poisson algebra in the sense of algebraic operads.
A DPP algebra contains a perm algebra and a Leibniz algebra such that some
compatibility conditions are satisfied, which is a special DPP algebra
introduced by Loday \cite{Lod}. Each Poisson algebra can be viewed as a special
DPP algebra. There is a natural construction of DPP algebras from Poisson algebras
equipped with average operators \cite{Agu}. The operad of DPP algebras is the
duplicator of the operad of Poisson algebras \cite{PBGN}. Recently, Lu has proved
that the operad of DPP algebras is Koszul, and shown that DPP algebras are the
semi-classical limits of diassociative formal deformations of perm algebras \cite{Lu}.

A bialgebra structure on a given algebra structure is obtained by a comultiplication
together with some compatibility conditions between the multiplication and the
comultiplication. A famous example of bialgebras is Lie bialgebras \cite{Dri},
which can be viewed as a linearization of Poisson-Lie groups.
In recent years, the bialgebra theories of various algebra structures have been
extensive developed, such as left-symmetric bialgebras \cite{Bai}, antisymmetric
infinitesimal bialgebras \cite{Bai1}, Novikov bialgebras \cite{HBG}, perm bialgebras
\cite{Hou,LZB}, Poisson bialgebras \cite{NB}, etc. Pre-Poisson bialgebras were
studied in \cite{WS} and dual pre-Poisson bialgebras (DPP bialgebras) were
introduced in \cite{Lu}. In this paper, we further investigate the bialgebra
theory of DPP algebras, especially quasi-triangular, triangular and factorizable
DPP bialgebras.

In the Lie bialgebra theory, quasi-triangular Lie bialgebras especially triangular Lie
bialgebras play important roles in mathematical physics \cite{Sem,BGN,RS}.
Factorizable Lie bialgebras constitute a special subclass of quasi-triangular
Lie bialgebras that bridge classical $r$-matrices with certain factorization problems,
which is found diverse applications in integrable systems \cite{Dri,LS}.
Recently, quasi-triangular bialgebra structure, especially
factorizable bialgebra structure, have received a lot of attention. Factorizable Lie
bialgebras \cite{LS}, factorizable antisymmetric infinitesimal bialgebras \cite{SW},
factorizable Leibniz bialgebras \cite{BLST}, factorizable Novikov
bialgebras \cite{CH}, factorizable perm bialgebras \cite{Lin}
and factorizable Poisson bialgebras \cite{LL,GH} have been further studied.

In this paper, we first introduce an invariant condition for a 2-tensor and notion of
the (classical) dual pre-Poisson Yang-Baxter equation ($\DPYBE$) in a DPP algebra,
by which we successfully define quasi-triangular DPP bialgebras and factorizable DPP
bialgebras. A factorizable DPP bialgebra induces a factorization of the underlying
DPP algebra and the double of any DPP bialgebra naturally admits a factorizable DPP
bialgebra structure. Furthermore, we introduce the notion of quadratic Rota-Baxter DPP
algebras of weights, demonstrate that a quadratic Rota-Baxter DPP algebra of weight
zero induces a triangular DPP bialgebra, and establish a one-to-one correspondence between
quadratic Rota-Baxter DPP algebras of nonzero weights and factorizable DPP bialgebras.
We summarize these results in the following diagram:
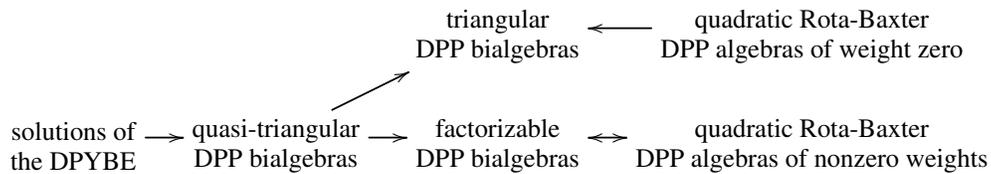
\begin{equation*}
\xymatrix@C=0.5cm@R=0.5cm{
&&\txt{\footnotesize triangular \\ \footnotesize DPP bialgebras}
&\txt{\footnotesize quadratic Rota-Baxter \\ 
\footnotesize DPP algebras of weight zero} \ar[l]\\
\txt{\footnotesize solutions of \\ \footnotesize the $\DPYBE$ }\ar[r]
& \txt{\footnotesize quasi-triangular \\ \footnotesize DPP bialgebras}\ar[r]\ar[ur]
&\txt{\footnotesize factorizable \\ \footnotesize DPP bialgebras}\ar@{<->}[r]
&\txt{\footnotesize quadratic Rota-Baxter \\
\footnotesize DPP algebras of nonzero weights}}
\end{equation*}

Note that, the operad of DPP algebras, being quadratic binary, is isomorphic to the
Manin white product of the operad of perm algebras and the operad of Poisson algebras
\cite{PBGN,Uch}. We elevate this conclusion to the level of bialgebras and give a
construction of infinite-dimensional DPP bialgebras. More precisely, we introduce
a class of infinite-dimensional DPP bialgebras, the completed DPP bialgebras.
We prove that there is a completed DPP bialgebra structure on the tensor product
Poisson bialgebra and a quadratic $\bz$-graded perm algebra. Moreover, by discussing
the relationship between solutions of the $\DPYBE$ in the induced DPP algebra and
solutions of the $\PoiYBE$ in the original Poisson algebra, we can conclude that this
construction preserves many properties of bialgebras, such as the coboundary (resp.
quasi-triangular, triangular) structure.

\smallskip\noindent
{\bf Theorem } (Theorems \ref{thm:P+perm=dP} and \ref{thm:indu-DPbia})
{\it Let $(P, \cdot, [-,-], \Delta, \delta)$ be a finite-dimensional Poisson bialgebra and
$(B=\oplus_{i\in\bz}B_{i}, \circ_{B}, \omega)$ be a quadratic $\bz$-graded perm algebra.
Then, there is a completed DPP bialgebra structure on the tensor product
on $P\otimes B$. which is denoted by $(P\otimes B, \circ, \ast, \nu, \vartheta)$.

In particular, the completed DPP bialgebra $(P\otimes B, \circ, \ast, \nu,
\vartheta)$ is coboundary (resp. quasi-triangular, triangular) if
$(P, \cdot, [-,-], \Delta, \delta)$ is coboundary (resp. quasi-triangular, triangular).}

\smallskip
Moreover, if the quadratic $\bz$-graded perm algebra $(B=\oplus_{i\in\bz}B_{i},
\circ_{B}, \omega)$ is just a finite-dimensional quadratic perm algebra,
we can also get that the DPP bialgebra $(P\otimes B, \circ, \ast, \nu,
\vartheta)$ is factorizable if $(P, \cdot, [-,-], \Delta, \delta)$ is factorizable.
And in this situation, we can derive an $\mathcal{O}$-operator on induced DPP algebra
$(P\otimes B, \circ, \ast)$ from an $\mathcal{O}$-operator on Poisson bialgebra
$(P, \cdot, [-,-])$.

The paper is organized as follows.
In Section \ref{sec:poidia}, we recall the notions of DPP algebras and
representations of DPP algebras, and show that there is a DPP algebra structure on
the tensor product of a Poisson algebra and a perm algebra.
In Section \ref{sec:bialg}, we recall the notion the $\DPYBE$ in a DPP algebra,
and introduce the notions of quasi-triangular DPP bialgebras, triangular DPP bialgebras
and factorizable DPP bialgebras. Each symmetric solution of the $\DPYBE$ induces a
triangular DPP bialgebra. The double of any DPP bialgebra is factorizable (see Proposition
\ref{pro:fact-doub}). Moreover, we introduce the notion of quadratic Rota-Baxter DPP
algebras, and establish a one-to-one correspondence between quadratic Rota-Baxter DPP
algebras of nonzero weights and factorizable DPP bialgebras in Theorem \ref{thm:fpb-qrPd}.
In Section \ref{sec:poi-di}, we show in Theorem \ref{thm:P+perm=dP} that the tensor
product of a finite-dimensional Poisson bialgebra and a quadratic $\bz$-graded perm
algebra can be naturally endowed with a completed DPP bialgebra. Constructions of some
special completed solutions of the $\DPYBE$ in the induced $\bz$-graded DPP algebra
from the solutions of the $\PoiYBE$ in the original Poisson algebra are obtained in
Proposition \ref{pro:PYBE-DPYBE}. Therefore, in Theorem \ref{thm:indu-DPbia}, we show
that the induced completed DPP bialgebra is coboundary (resp. quasi-triangular, triangular)
if the original Poisson bialgebra is coboundary (resp. quasi-triangular, triangular).
In particular, the factorizablility of bialgebra and the $\mathcal{O}$-operators
are discussed whenever the induced DPP bialgebra is finite-dimensional.

Throughout this paper, we fix $\Bbbk$ as a field of characteristic zero.
All the vector spaces, algebras are over $\Bbbk$ and are finite-dimensional
unless otherwise specified, and all tensor products are also over $\Bbbk$.
We denote the identity map by $\id$. For any finite-dimensional $\Bbbk$-vector
space $V$, we denote $V^{\ast}$ the dual space.

\bigskip

\section{Dual pre-Poisson algebras and their representations} \label{sec:poidia}
In this section, we recall the notions of dual pre-Poisson algebras and
representations of dual pre-Poisson algebras. Recall that a (left) {\bf Leibniz algebra}
$(A, \ast)$ is a vector space $A$ together with a bilinear map $\ast: A\otimes
A\rightarrow A$ satisfying the following Leibniz identity:
$$
a_{1}\ast(a_{2}\ast a_{3})=(a_{1}\ast a_{2})\ast a_{3}+a_{2}\ast(a_{1}\ast a_{3}),
$$
for any $a_{1}, a_{2}, a_{3}\in A$. A {\bf perm algebra} is a pair $(A, \circ)$,
where $A$ is a vector space and $\circ: A\otimes A\rightarrow A$ is a bilinear
operator such that for any $a_{1}, a_{2}, a_{3}\in A$,
$$
a_{1}\circ(a_{2}\circ a_{3})=(a_{1}\circ a_{2})\circ a_{3}=(a_{2}\circ a_{1})\circ a_{3},
$$
In a perm algebra $(A, \circ)$, we also have
$a_{1}\diamond(a_{2}\diamond a_{3})=a_{2}\diamond(a_{1}\diamond a_{3})$.

\begin{defi}[\cite{Agu}]\label{def:dip-alg}
A {\bf dual pre-Poisson algebra} (DPP algebra) $(A, \circ, \ast)$ is a vector space $A$
together with two bilinear maps $\circ, \ast: A\otimes A\rightarrow A$ such that
$(A, \circ)$ is a perm algebra, $(A, \ast)$ is a Leibniz algebra and the following
equalities hold for all $a_{1}, a_{2}, a_{3}\in A$,
\begin{align*}
&\qquad\quad (a_{1}\circ a_{2})\ast a_{3}=a_{1}\circ(a_{2}\ast a_{3})
+a_{2}\circ(a_{1}\ast a_{3}),\\
&(a_{1}\ast a_{2})\circ a_{3}=a_{1}\ast(a_{2}\circ a_{3})-a_{2}\circ(a_{1}\ast a_{3})
=-(a_{2}\ast a_{1})\circ a_{3}.
\end{align*}
\end{defi}

In a DPP algebra $(A, \circ, \ast)$, we also have $(a_{1}\circ a_{2})\ast a_{3}
=(a_{2}\circ a_{1})\ast a_{3}$ and $a_{1}\ast(a_{2}\circ a_{3})-a_{2}\circ(a_{1}\ast a_{3})
=a_{1}\circ(a_{2}\ast a_{3})-a_{2}\ast(a_{1}\circ a_{3})$ for any $a_{1}, a_{2}, a_{3}\in A$.
Let $(A, \circ, \ast)$ and $(A', \circ', \ast')$ be two DPP algebras. A linear
map $f: A\rightarrow A'$ is called a {\bf homomorphism of DPP algebras}
if $f(a_{1}\circ a_{2})=f(a_{1})\circ' f(a_{2})$ and $f(a_{1}\ast a_{2})=f(a_{1})\ast'
f(a_{2})$ for any $a_{1}, a_{2}\in A$.

\begin{ex}\label{ex:alg}
Let $A$ be a 2-dimension vector space with $\{e_{1}, e_{2}\}$ a base.
We defined the multiplications $\circ$ and $\ast$ by $e_{2}\ast e_{2}=e_{1}=
e_{2}\circ e_{2}$. Then $(A, \circ, \ast)$ is a 2-dimension DPP algebra.
\end{ex}

Next, we consider the representations of a DPP algebra. Recall that a
{\bf representation $(V, \kl, \kr)$ of a perm algebra} $(A, \circ)$ is a vector space
$V$ with two linear maps $\kl, \kr: A\rightarrow\gl(V)$ such that
\begin{align*}
&\qquad\; \kl(a_{1}\circ a_{2})=\kl(a_{1})\kl(a_{2})=\kl(a_{2})\kl(a_{1}),\\
& \kr(a_{1}\circ a_{2})=\kr(a_{2})\kr(a_{1})=\kr(a_{2})\kl(a_{1})=\kl(a_{1})\kr(a_{2}),
\end{align*}
for any $a_{1}, a_{2}\in A$. Let $(A, \ast)$ be a Leibniz algebra, $V$ be a vector space
and $\kkl, \kkr: A\rightarrow\gl(V)$ be two linear maps. Then $(V, \kkl, \kkr)$ is
called a {\bf representation of $(A, \ast)$} if for any $a_{1}, a_{2}\in A$,
\begin{align*}
&\qquad\; \kkl(a_{1}\ast a_{2})=\kkl(a_{1})\kkl(a_{2})-\kkl(a_{2})\kkl(a_{1}),\\
&\kkr(a_{1})\kkr(a_{2})=\kkr(a_{1}\ast a_{2})+\kkl(a_{2})\kkr(a_{1})=-\kkr(a_{1})\kkl(a_{2}).
\end{align*}

\begin{defi}\label{def:dprep}
Let $(A, \circ, \ast)$ be a DPP algebra, $V$ be a vector space and $\kl, \kr, \kkl,
\kkr: A\rightarrow\gl(V)$ be four linear maps. If $(V, \kl, \kr)$ is a representation of
$(A, \circ)$, $(V, \kkl, \kkr)$ is a representation of $(A, \ast)$ and the for any
$a_{1}, a_{2}\in A$,
\begin{align*}
& \qquad\qquad\qquad \kkl(a_{1}\circ a_{2})=\kl(a_{1})\kkl(a_{2})+\kl(a_{2})\kkl(a_{1}),\\
& \qquad\qquad \kkr(a_{2})\kl(a_{1})=\kl(a_{1})\kkr(a_{2})+\kr(a_{1}\ast a_{2})
=\kkr(a_{2})\kr(a_{1}),\\
& \qquad\qquad \kl(a_{1}\ast a_{2})=\kkl(a_{1})\kl(a_{2})-\kl(a_{2})\kkl(a_{1})
=-\kl(a_{2}\ast a_{1}),\\
& \kr(a_{2})\kkr(a_{1})=\kkr(a_{1}\circ a_{2})-\kl(a_{1})\kkr(a_{2})
=-\kr(a_{2})\kkl(a_{1})=\kr(a_{1}\ast a_{2})-\kkl(a_{1})\kr(a_{2}),
\end{align*}
then $(V, \kl, \kr, \kkl, \kkr)$ is called a {\bf representation of the DPP algebra}
$(A, \circ, \ast)$.
\end{defi}

For the representations of DPP algebras, we have the following equivalent
characterization.

\begin{pro}\label{pro:rep-dipoi}
Let $(A, \circ, \ast)$ be a DPP algebra, $V$ be a vector space and $\kl, \kr, \kkl,
\kkr: A\rightarrow\gl(V)$ be four linear maps. Define bilinear maps $\tilde{\circ},
\tilde{\ast}: (A\oplus V)\otimes(A\oplus V)\rightarrow A\oplus V$ by
\begin{align}
(a_{1}, v_{1})\tilde{\circ}(a_{2}, v_{2})&:=\big(a_{1}\circ a_{2},\ \
\kl(a_{1})(v_{2})+\kr(a_{2})(v_{1})\big), \label{sp1}\\
(a_{1}, v_{1})\tilde{\ast}(a_{2}, v_{2})&:=\big(a_{1}\ast a_{2},\ \
\kkl(a_{1})(v_{2})+\kkr(a_{2})(v_{1})\big),  \label{sp2}
\end{align}
for any $a_{1}, a_{2}\in A$ and $v_{1}, v_{2}\in V$. Then $(V, \kl, \kr, \kkl, \kkr)$
is a representation of $(A, \circ, \ast)$ if and only if $(A\oplus V, \tilde{\circ},
\tilde{\ast})$ is a DPP algebra. We denote this DPP algebra by
$A\ltimes V$, and call it the {\bf semidirect product} of $(A, \circ, \ast)$
by representation $(V, \kl, \kr, \kkl, \kkr)$.
\end{pro}

\begin{proof}
It is a straightforward check.
\end{proof}

Let $(V, \kl, \kr, \kkl, \kkr)$ and $(V', \kl', \kr', \kkl', \kkr')$ be two representations
of a DPP algebra $(A, \circ, \ast)$. A linear map $f: V\rightarrow V'$ is called
a {\bf morphism of representations} if $f(\beta(a)(v))=\beta'(a)(f(v))$ for any
$a\in A$, $v\in V$ and $\beta\in\{\kl, \kr, \kkl, \kkr\}$. A morphism of representations
$f$ is said to be an {\bf isomorphism} if $f$ is a bijection. Define the left multiplication
maps $\fl_{A}, \ffl_{A}: A\rightarrow\gl(A)$ and the right multiplication maps
$\fr_{A}, \ffr_{A}: A\rightarrow\gl(A)$ by $\fl_{A}(a_{1})(a_{2})=a_{1}\circ a_{2}$,
$\ffl_{A}(a_{1})(a_{2})=a_{1}\ast a_{2}$, $\fr_{A}(a_{1})(a_{2})=a_{2}\circ a_{1}$ and
$\ffr_{A}(a_{1})(a_{2})=a_{2}\ast a_{1}$ respectively for all $a_{1}, a_{2}\in A$.
Then $(A, \fl_{A}, \fr_{A}, \ffl_{A}, \ffr_{A})$ is a representation of $(A, \circ, \ast)$,
which is called the {\bf regular representation} of $(A, \circ, \ast)$.

Let $V$, $W$ be two finite-dimensional $\Bbbk$-vector spaces. We denote
$\langle-,-\rangle$ the natural pairing between the spaces $V$ and $V^{\ast}$,
i.e., $\langle\xi, v\rangle:=\xi(v)\in\Bbbk$ for any $v\in V$ and $\xi\in V^{\ast}$.
For a linear map $f: V\rightarrow W$, we define the map $f^{\ast}: W^{\ast}
\rightarrow V^{\ast}$ by $\langle f^{\ast}(\xi), v\rangle
=\langle\xi, f(v)\rangle$ for any $v\in V$ and $\xi\in W^{\ast}$.
Let $(A, \circ, \ast)$ be a finite-dimensional DPP algebra.
For each $\beta\in\{\fl_{A}, \fr_{A}, \ffl_{A}, \ffr_{A}\}$,
we define a linear map $\beta^{\ast}: A\rightarrow\gl(A^{\ast})$ by
$$
\langle\beta^{\ast}(a_{1})(\xi),\; a_{2}\rangle
=-\langle\xi,\; \beta(a_{1})(a_{2})\rangle,
$$
for any $a_{1}, a_{2}\in A$ and $\xi\in A^{\ast}$.
For a DPP algebra $(A, \circ, \ast)$, one can check that $(A^{\ast},
\fl_{A}^{\ast}, \fl_{A}^{\ast}-\fr_{A}^{\ast}, \ffl_{A}^{\ast},
-\ffl_{A}^{\ast}-\ffr_{A}^{\ast})$ of $(A, \circ, \ast)$,
which is called the {\bf coregular representation} of $(A, \circ, \ast)$.

Let $\omega(-, -)$ be a bilinear form on a DPP algebra $(A, \circ, \ast)$.
Recall that
\begin{itemize}
\item[-] $\omega(-, -)$ is called {\bf nondegenerate} if
     $\omega(a_{1}, a_{2})=0$ for any $a_{2}\in A$, then $a_{1}=0$;
\item[-] $\omega(-, -)$ is called {\bf skew-symmetric} if $\omega(a_{1}, a_{2})
     =-\omega(a_{2}, a_{1})$, for any $a_{1}, a_{2}\in A$;
\item[-] $\omega(-, -)$ is called {\bf invariant} if $\omega(a_{1}\circ a_{2},\; a_{3})
     =\omega(a_{1},\; a_{2}\circ a_{3}-a_{3}\circ a_{2})$ and $\omega(a_{1}\ast a_{2},\;
     a_{3})=\omega(a_{1},\; a_{2}\ast a_{3}+a_{3}\ast a_{2})$ for any $a_{1}, a_{2},
     a_{3}\in A$.
\end{itemize}
A DPP algebra $(A, \circ, \ast)$ with a nondegenerate skew-symmetric invariant
bilinear form $\omega(-, -)$ is called a {\bf quadratic DPP algebra} and
denoted by $(A, \circ, \ast, \omega)$. It is easy to see that,
$\omega(a_{1}\circ a_{2},\; a_{3})=\omega(a_{2},\; a_{1}\circ a_{3})$ and
$\omega(a_{1}\ast a_{2},\; a_{3})=-\omega(a_{2},\; a_{1}\ast a_{3})$ for any
$a_{1}, a_{2}, a_{3}\in A$ if $(A, \circ, \ast, \omega)$ is quadratic.
Moreover, one can check that $(A, \fl_{A}, \fr_{A}, \ffl_{A}, \ffr_{A})$
and $(A^{\ast}, -\fl_{A}^{\ast}, \fr_{A}^{\ast}-\fl_{A}^{\ast}, \ffl_{A}^{\ast},
-\ffl_{A}^{\ast}-\ffr_{A}^{\ast})$ are isomorphic as representations of
the DPP algebra $(A, \circ, \ast)$ if and only if there exists a
nondegenerate skew-symmetric invariant bilinear form on $(A, \circ, \ast)$.

There is a close relationship between DPP algebra algebra and Poisson algebra.
We first give the definition of Poisson algebra. Recall that a {\bf commutative associative
algebra} $(A, \cdot)$ is a vector space $A$ with a bilinear map $\cdot: A\otimes A
\rightarrow A$ such that $a_{1}a_{2}=a_{2}a_{1}$ and $a_{1}(a_{2}a_{3})=(a_{1}a_{2})a_{3}$
for any $a_{1}, a_{2}, a_{3}\in A$, where $a_{1}a_{2}:=a_{1}\cdot a_{2}$ for brevity.
A {\bf Lie algebra} $(\g, [-,-])$ is a vector space $\g$ with an antisymmetric bilinear
bracket $[-,-]: \g\otimes\g\rightarrow\g$ such that $[g_{1}, [g_{2}, g_{3}]]
+[g_{2}, [g_{3}, g_{1}]]+[g_{3}, [g_{1}, g_{2}]]=0$ for any $g_{1}, g_{2}, g_{3}\in\g$.

\begin{defi}\label{def:p-alg}
A {\bf Poisson algebra} is a triple $(P, \cdot, [-,-])$, where $(P, \cdot)$ is a
commutative associative algebra and $(P, [-,-])$ is a Lie algebra, such that the
Leibniz rule holds:
$$
[p_{1},\; p_{2}p_{3}]=[p_{1}, p_{2}]p_{3}+p_{2}[p_{1}, p_{3}],
$$
for any $p_{1}, p_{2}, p_{3}\in P$.
\end{defi}

Like other types of dialgebra, Poisson algebra can be seen as a special type of
DPP algebra.
Let $(P, \cdot, [-,-])$ be a Poisson algebra. A linear map $\alpha: P\rightarrow P$
is called an {\bf averaging operator} if for any $p_{1}, p_{2}\in P$,
\begin{align*}
&\quad [\alpha(p_{1}),\; \alpha(p_{2})]=\alpha([\alpha(p_{1}),\; p_{2}]),\\
&\alpha(p_{1})\alpha(p_{2})=\alpha(\alpha(p_{1})p_{2})=\alpha(p_{1}\alpha(p_{2})).
\end{align*}
If $\alpha: P\rightarrow P$ is an averaging operator on a Poisson algebra $(P, \cdot, [-,-])$,
we have $[\alpha(p_{1}),\; \alpha(p_{2})]=\alpha([\alpha(p_{1}),\; p_{2}])=
\alpha([p_{1},\; \alpha(p_{2})])$. In \cite{Agu}, the authors have shown that
a Poisson algebra with an averaging operator induces a DPP algebra.

\begin{pro}[\cite{Agu}]\label{pro:avpoi}
Let $(P, \cdot, [-,-])$ be a Poisson algebra and $\alpha: P\rightarrow P$ be an averaging
operator. Define bilinear maps $\circ, \ast: P\otimes P\rightarrow P$ by
$$
p_{1}\circ p_{2}:=\alpha(p_{1})p_{2} \qquad\mbox{ and }\qquad
p_{1}\ast p_{2}:=[\alpha(p_{1}), p_{2}],
$$
for any $p_{1}, p_{2}\in P$. Then $(P, \circ, \ast)$ is a DPP algebra.
\end{pro}

Moreover, we can get a DPP algebra by the tensor product of a Poisson algebra and
a perm algebra.

\begin{pro}[\cite{Lu}]\label{pro:poi+perm}
Let $(P, \cdot, [-,-])$ be a Poisson algebra and $(B, \circ_{B})$ be a perm algebra.
Define bilinear maps $\circ, \ast: (P\otimes B)\otimes(P\otimes B)\rightarrow P\otimes B$ by
\begin{align*}
(p_{1}, b_{1})\circ(p_{2}, b_{2})&:=(p_{1}p_{2})\otimes(b_{1}\circ_{B} b_{2}),\\
(p_{1}, b_{1})\ast(p_{2}, b_{2})&:=[p_{1}, p_{2}]\otimes(b_{1}\circ_{B} b_{2}),
\end{align*}
for any $p_{1}, p_{2}\in P$ and $b_{1}, b_{2}\in B$.
Then $(P\otimes B, \circ, \ast)$ is a DPP algebra.
\end{pro}

\begin{proof}
First, it is easy to see that $(P\otimes B, \circ)$ is a perm algebra and
$(P\otimes B, \ast)$ is a Leibniz algebra. Second, for any $p_{1}, p_{2}, p_{3}\in P$
and $b_{1}, b_{2}, b_{3}\in B$, we have
\begin{align*}
((p_{1}, b_{1})\ast(p_{2}, b_{2}))\circ(p_{3}, b_{3})
&=([p_{1}, p_{2}]p_{3})\otimes((b_{1}\circ_{B} b_{2})\circ_{B} b_{3})\\
&=-([p_{2}, p_{1}]p_{3})\otimes((b_{2}\circ_{B} b_{1})\circ_{B} b_{3})\\
&=-((p_{2}, b_{2})\ast(p_{1}, b_{1}))\circ(p_{3}, b_{3}).
\end{align*}
Similarly, we also have $(p_{1}, b_{1})\ast((p_{2}, b_{2})\circ(p_{3}, b_{3}))
=((p_{1}, b_{1})\ast(p_{2}, b_{2}))\circ(p_{3}, b_{3})+(p_{2}, b_{2})\circ((p_{1}, b_{1})
\ast(p_{3}, b_{3}))$ and $((p_{1}, b_{1})\circ(p_{2}, b_{2}))\ast(p_{3}, b_{3})
=(p_{1}, b_{1})\circ((p_{2}, b_{2})\ast(p_{3}, b_{3}))+(p_{2}, b_{2})\circ((p_{1}, b_{1})
\ast(p_{3}, b_{3}))$. Thus, $(P\otimes B, \circ, \ast)$ is a DPP algebra.
\end{proof}

In this paper, we not only study some special DPP bialgebras associated with
solutions of the classical Yang-Baxter equation, but also elevate the above relationship
between Poisson algebras and DPP algebras to the level of bialgebra.

\bigskip
\section{Quasi-triangular DPP bialgebras and classical Yang-Baxter equation}
\label{sec:bialg}
Recently, Lu have given a bialgebra theory of DPP algebras \cite{Lu}. Here we
further study some special DPP bialgebras associated with solutions of
the classical Yang-Baxter equation in a DPP algebra by introducing the
notions of quasi-triangular DPP bialgebras, triangular DPP bialgebras
and factorizable Poisson bidialgebras. First, we recall the definition of dual
pre-Poisson coalgebras. A {\bf perm coalgebra} $(A, \nu)$ is a vector space $A$
with a linear map $\nu: A\rightarrow A\otimes A$ satisfying
$$
(\nu\otimes\id)\nu=(\id\otimes\nu)\nu=(\tau\otimes\id)(\id\otimes\nu)\nu.
$$
One can check that $(A, \nu)$ is a perm coalgebra if and only if
$(A^{\ast}, \nu^{\ast})$ is a perm algebra. A {\bf Leibniz coalgebra} is a pair
$(A, \vartheta)$, where $B$ is a vector space and $\vartheta: B\rightarrow B \otimes B$
is a linear map satisfying
$$
(\vartheta\otimes\id)\vartheta=(\id\otimes\vartheta)\vartheta
+(\id\otimes\tau)(\vartheta\otimes\id)\vartheta.
$$

\begin{defi}\label{def:p-coalg}
A {\bf dual pre-Poisson coalgebra} (DPP coalgebra) is a triple $(A, \nu, \vartheta)$,
where $(A, \nu)$ is a perm coalgebra, $(A, \vartheta)$ is a Leibniz coalgebra and
the following compatibility conditions hold:
\begin{align*}
(\id\otimes\nu)\vartheta&=(\vartheta\otimes\id)\nu+(\tau\otimes\id)(\id\otimes\vartheta)\nu,\\
(\nu\otimes\id)\vartheta&=(\id\otimes\vartheta)\nu+(\tau\otimes\id)(\id\otimes\vartheta)\nu,\\
(\vartheta\otimes\id)\nu&=-(\tau\otimes\id)(\vartheta\otimes\id)\nu.
\end{align*}
\end{defi}

By direct calculation, one can show that $(A, \nu, \vartheta)$ is a DPP coalgebra
if and only $(A^{\ast}, \nu^{\ast}, \vartheta^{\ast})$ is a DPP algebra.
Now we give the definition of DPP bialgebras. Recall that a {\bf perm
bialgebra} is a triple $(A, \circ, \nu)$, where $(A, \circ)$ is a perm algebra,
$(A, \nu)$ is a perm coalgebra and for any $a_{1}, a_{2}\in A$,
\begin{align*}
&\qquad\qquad (\fr_{A}(a_{1})\otimes\id)\nu(a_{2})
=\tau((\fr_{A}(a_{2})\otimes\id)\nu(a_{1})),   \\
&\qquad \nu(a_{1}\circ a_{2})=((\fl_{A}-\fr_{A})(a_{1})\otimes\id)\nu(a_{2})
+(\id\otimes\,\fr_{A}(a_{2}))\nu(a_{1}),  \\
&\nu(a_{1}\circ a_{2})=(\id\otimes\,\fl_{A}(a_{1}))\nu(a_{2})
+((\fl_{A}-\fr_{A})(a_{2})\otimes\id)(\nu(a_{1})-\tau(\nu(a_{1})).
\end{align*}
A {\bf Leibniz bialgebra} is a triple $(A, \ast, \vartheta)$, where $(A, \ast)$ is a
Leibniz algebra, $(A, \vartheta)$ is a Leibniz coalgebra and the following equations hold:
\begin{align*}
&\qquad\qquad\quad \tau((\ffr_{A}(a_{2})\otimes\id)(\vartheta(a_{1})))
=(\ffr_{A}(a_{1})\otimes\id)(\vartheta(a_{2})),\\
&\vartheta(a_{1}\ast a_{2})=(\id\otimes\ffr_{A}(a_{2})-(\ffl_{A}+\ffr_{A})(a_{2})\otimes\id)
((\id\otimes\id+\tau)(\vartheta(a_{1})))\\[-1mm]
&\qquad\qquad\qquad\qquad\qquad\qquad\qquad +(\id\otimes\ffl_{A}(a_{1})
+\ffl_{A}(a_{1})\otimes\id)(\vartheta(a_{2})),
\end{align*}
for any $a_{1}, a_{2}\in A$.

\begin{defi}[\cite{Lu}]\label{def:bi-dia}
A {\bf dual pre-Poisson bialgebra (DPP bialgebra)} is a quituple $(A, \circ, \ast,
\nu, \vartheta)$, where $(A, \circ, \ast)$ is a DPP algebra, $(A, \nu, \vartheta)$
is a DPP coalgebra, $(A, \circ, \nu)$ is a perm bialgebra, $(A, \ast, \vartheta)$
is a Leibniz bialgebra and the following compatibility conditions hold:
{\small\begin{align}
&\nu(a_{1}\ast a_{2})=(\id\otimes\ffl_{A}(a_{1})+\ffl_{A}(a_{1})\otimes\id)(\nu(a_{2}))
+((\fl_{A}-\fr_{A})(a_{2})\otimes\id-\id\otimes\fr_{A}(a_{2}))
((\vartheta+\tau(\vartheta))(a_{1})),    \label{dpbi1} \\
&\qquad\vartheta(a_{1}\circ a_{2})=(\id\otimes\fl_{A}(a_{1}))(\vartheta(a_{2}))
+(\id\otimes\fr_{A}(a_{2}))(\vartheta(a_{1}))-((\ffl_{A}+\ffr_{A})(a_{1})\otimes\id)(\nu(a_{2}))
\label{dpbi2} \\[-1mm]
&\qquad\qquad\qquad\qquad-((\ffl_{A}+\ffr_{A})(a_{2})\otimes\id )((\nu-\tau(\nu))(a_{1})),
\nonumber\\
&\qquad\qquad\qquad\qquad(\id\otimes\ffr_{A}(a_{1}))(\tau(\nu(a_{2})))
=-(\fr_{A}(a_{2})\otimes\id)(\vartheta(a_{1})),
\label{dpbi3}
\end{align}
\begin{align}
&\nu(a_{1}\ast a_{2})=(\id\otimes\fl_{A}(a_{1})-\fl_{A}(a_{1})\otimes\id)(\vartheta(a_{2}))
+(\id\otimes\ffr_{A}(a_{2})-(\ffl_{A}+\ffr_{A})(a_{2})\otimes\id)((\nu-\tau(\nu))(a_{1})),
\label{dpbi4} \\
&\qquad (\vartheta+\tau(\vartheta))(a_{1}\circ a_{2})=(\id\otimes\fl_{A}(a_{1}))((\vartheta
+\tau(\vartheta))(a_{2}))+(\id\otimes\fl_{A}(a_{2}))((\vartheta+\tau(\vartheta))(a_{1}))
\label{dpbi5} \\[-1mm]
&\qquad\qquad\qquad\qquad\qquad\qquad-(\ffl_{A}(a_{1})\otimes\id)((\nu-\tau(\nu))(a_{2}))
-(\ffl_{A}(a_{2})\otimes\id)((\nu-\tau(\nu))(a_{1})), \nonumber\\
&\qquad\qquad\vartheta(a_{1}\circ a_{2})=(\fl_{A}(a_{1})\otimes\id)(\vartheta(a_{2}))
-(\id\otimes\ffr_{A}(a_{2}))((\nu-\tau(\nu))(a_{1}))
\label{dpbi6} \\[-1mm]
&\qquad\qquad\qquad\qquad\qquad+(\id\otimes\ffl_{A}(a_{1}))(\nu(a_{2}))
+((\fl_{A}-\fr_{A})(a_{2})\otimes\id)((\vartheta+\tau(\vartheta))(a_{1})), \nonumber\\
&\qquad\nu(a_{1}\ast a_{2}+a_{2}\ast a_{1})=(\id\otimes(\ffl_{A}+\ffr_{A})(a_{1}))(\nu(a_{2}))
+((\ffl_{A}+\ffr_{A})(a_{1})\otimes\id)(\tau(\nu(a_{2})))
\label{dpbi7} \\[-1mm]
&\qquad\qquad\qquad\qquad\qquad\qquad+(\id\otimes(\fl_{A}-\fr_{A})(a_{2}))(\vartheta(a_{1}))
+((\fl_{A}-\fr_{A})(a_{2})\otimes\id)(\tau(\vartheta(a_{1}))). \nonumber
\end{align}}
\end{defi}

In \cite{Lu}, Lu have introduced the notions of matched pairs and Manin triples of DPP
algebras, and the equivalence between matched pairs of DPP algebras, Manin triples of
DPP algebras and DPP bialgebras is interpreted. Here we main study some
special DPP bialgebras. Let $(A, \ast)$ be a Leibniz algebra. We define a linear map
$F: A\rightarrow\gl(A\otimes A)$ by
$$
F(a):=(\ffl_{A}+\ffr_{A})(a)\otimes\id)-\id\otimes\ffr_{A}(a).
$$
An element $r\in A\otimes A$ is called {\bf Leib-invariant} if $F(a)(r)=0$ for all $a\in A$.
If there exists an element $r\in A\otimes A$ such that $(A, \ast, \vartheta_{r})$ is a
Leibniz bialgebra, where $\vartheta_{r}: A\rightarrow A\otimes A$ is given by
\begin{align}
\vartheta_{r}(a)=F(a)(r),    \label{cobLeb}
\end{align}
for any $a\in A$, then $(A, \ast, \vartheta_{r})$ is called a {\bf coboundary
Leibniz bialgebra} associated with $r$. Let $(A, \ast)$ be a Leibniz algebra and
$r=\sum_{i}x_{i}\otimes y_{i}\in A\otimes A$. The equation
$$
\mathbf{L}_{r}=r_{12}\ast r_{13}+r_{12}\ast r_{23}-r_{23}\ast r_{12}+r_{23}\ast r_{13}=0
$$
is called the (classical) {\bf Leibniz Yang-Baxter equation} ($\LYBE$) in the Leibniz
algebra $(A, \ast)$, where $r_{12}\ast r_{13}=\sum_{i,j}(x_{i}\ast x_{j})\otimes y_{i}
\otimes y_{j}$, $r_{12}\ast r_{23}=\sum_{i,j}x_{i}\otimes(y_{i}\ast x_{j})\otimes y_{j}$,
$r_{23}\ast r_{12}=\sum_{i,j}x_{j}\otimes(x_{i}\ast y_{j})\otimes y_{i}$ and
$r_{23}\ast r_{13}=\sum_{i,j}x_{j}\otimes x_{i}\otimes(y_{i}\ast y_{j})$.
Let $(A, \circ)$ be a perm algebra and $r=\sum_{i}x_{i}\otimes y_{i}\in A\otimes A$.
Then
$$
\mathbf{P}_{r}:=r_{12}\circ r_{23}-r_{13}\circ r_{23}+r_{12}\circ r_{13}-r_{13}\circ r_{12}=0
$$
is called the (classical) {\bf perm Yang-Baxter equation} ($\PYBE$) in $(A, \circ)$,
where $r_{12}\circ r_{23}:=\sum_{i,j}x_{i}\otimes(y_{i}\circ x_{j})\otimes y_{j}$,
$r_{13}\circ r_{23}:=\sum_{i,j}x_{i}\otimes x_{j}\otimes(y_{i}\circ y_{j})$,
$r_{12}\circ r_{13}:=\sum_{i,j}(x_{i}\circ x_{j})\otimes y_{i}\otimes y_{j}$,
$r_{13}\circ r_{12}:=\sum_{i,j}(x_{i}\circ x_{j})\otimes y_{j}\otimes y_{i}$.
An element $r\in A\otimes A$ is called {\bf perm-invariant} if
$$
G(a):=(\id\otimes\fr_{A}(a)+(\fr_{A}-\fl_{A})(a)\otimes\id)(r)=0
$$
for any $a\in A$. If there exists an element $r\in A\otimes A$ such that
$(A, \circ, \nu_{r})$ a perm bialgebra, where
\begin{align}
\nu_{r}(a)=G(a)(r),    \label{cobperm}
\end{align}
for any $a\in A$, then the perm bialgebra $(A, \circ, \nu_{r})$ is called
a {\bf coboundary perm bialgebra} associated with $r$.

Let $(A, \circ, \ast)$ be a DPP algebra. The equations $\mathbf{P}_{r}=\mathbf{L}_{r}=0$
is called the (classical) {\bf dual pre-Poisson Yang-Baxter equation} ($\DPYBE$) in
$(A, \circ, \ast)$. An element $r\in A\otimes A$ is called {\bf DP-invariant} if $F(a)(r)=
G(a)(r)=0$; is called {\bf symmetric} (resp. {\bf skew-symmetric}) if $r=\tau(r)$
(resp. $r=-\tau(r)$).

\begin{pro}[\cite{Lu}]\label{pro:sLib-bia}
Let $(A, \circ, \ast)$ be a DPP algebra, $r\in A\otimes A$, $\vartheta_{r},
\nu_{r}: A\rightarrow A\otimes A$ are given by Eqs. \eqref{cobLeb} and \eqref{cobperm}
respectively.
\begin{enumerate}
\item[$(i)$] If $r$ is a solution of the $\DPYBE$ in $(A, \circ, \ast)$ and
     $r-\tau(r)$ is DP-invariant, then $(A, \circ, \ast, \nu_{r}, \vartheta_{r})$
     is a DPP bialgebra, which is called a {\bf quasi-triangular DPP} associated with $r$.
\item[$(ii)$] If $r$ is a symmetric solution of the $\DPYBE$ in $(A, \circ,
     \ast)$, then $(A, \circ, \ast, \nu_{r}, \vartheta_{r})$ is a DPP bialgebra,
     which is called a {\bf triangular DPP bialgebra} associated with $r$.
\end{enumerate}
\end{pro}

Let $(A, \circ, \ast)$ be a DPP algebra. For any $r\in A\otimes A$,
we define a linear map $r^{\sharp}: A^{\ast}\rightarrow A$ by
$$
\langle r^{\sharp}(\xi_{1}),\; \xi_{2}\rangle=\langle\xi_{1}\otimes\xi_{2},\; r\rangle,
$$
for any $\xi_{1}, \xi_{2}\in A^{\ast}$ and denote $\mathcal{I}=r^{\sharp}-\tau(r)^{\sharp}: A^{\ast}\rightarrow A$.

\begin{defi}\label{def:fact-Leib}
Let $(A, \circ, \ast)$ be a DPP algebra, $r\in A\otimes A$ and $(A, \circ, \ast,
\nu_{r}, \vartheta_{r})$ be a quasi-triangular DPP bialgebra associated with $r$.
If $\mathcal{I}=r^{\sharp}-\tau(r)^{\sharp}: A^{\ast}\rightarrow A$ is an isomorphism
of vector spaces, then $(A, \circ, \ast, \nu_{r}, \vartheta_{r})$ is called a
{\bf factorizable DPP bialgebra}.
\end{defi}

By the definitions of factorizable perm bialgebra given in \cite{Lin} and factorizable
Leibniz bialgebra given in \cite{BLST}, It is easy to see that a quasi-triangular
DPP bialgebra $(A, \circ, \ast, \nu, \vartheta)$ is factorizable if and only
if $(A, \circ, \nu)$ as a perm bialgebra is factorizable or
$(A, \ast, \vartheta)$ as a Leibniz bialgebra is factorizable.
The factorizable DPP bialgebras exist in large quantities, and we can
construct a factorizable DPP bialgebra from any DPP bialgebra.
Let $(A, \circ, \ast, \nu, \vartheta)$ be a DPP bialgebra,
$\tilde{r}\in A\otimes A^{\ast}\subset (A\oplus A^{\ast})\otimes(A\oplus A^{\ast})$
corresponds to the identity map $\id: A\rightarrow A$. That is, $\tilde{r}=\sum_{i=1}^{n}
e_{i}\otimes f_{i}$, where $\{e_{1}, e_{2}, \cdots, e_{n}\}$ is a basis of $A$ and
$\{f_{1}, f_{2}, \cdots, f_{n}\}$ be its dual basis in $A^{\ast}$. Denote the dual
DPP algebra structure on $A^{\ast}$ of $(\nu, \vartheta)$ by $(\circ', \ast')$.
Suppose that $\nu_{\tilde{r}}, \vartheta_{\tilde{r}}: A\oplus A^{\ast}\rightarrow
(A\oplus A^{\ast})\otimes(A\oplus A^{\ast})$ are the linear maps defined by Eqs.
\eqref{cobperm} and \eqref{cobLeb} respectively. Then we obtain a factorizable
perm bialgebra $(A\oplus A^{\ast}, \tilde{\circ}, \nu_{\tilde{r}})$, where
$$
(a_{1}, \xi_{1})\tilde{\circ}(a_{2}, \xi_{2})
=\big(a_{1}\circ a_{2}+\fl_{A^{\ast}}^{\ast}(\xi_{1})(a_{2})
+\fr_{A^{\ast}}^{\ast}(\xi_{2})(a_{1}),\ \ \xi_{1}\circ'\xi_{2}
+\fl_{A}^{\ast}(a_{1})(\xi_{2})+\fr_{A}^{\ast}(a_{2})(\xi_{1})\big),
$$
for any $a_{1}, a_{2}\in A$ and $\xi_{1}, \xi_{2}\in A^{\ast}$
(see \cite[Theorem 3.6]{Lin}),
and a factorizable Leibniz bialgebra $(A\oplus A^{\ast}, \tilde{\ast},
\vartheta_{\tilde{r}})$, where
$$
(a_{1}, \xi_{1})\tilde{\ast}(a_{2}, \xi_{2})
=\big(a_{1}\ast a_{2}+\ffl_{A^{\ast}}^{\ast}(\xi_{1})(a_{2})
+\ffr_{A^{\ast}}^{\ast}(\xi_{2})(a_{1}),\ \ \xi_{1}\ast'\xi_{2}
+\ffl_{A}^{\ast}(a_{1})(\xi_{2})+\ffr_{A}^{\ast}(a_{2})(\xi_{1})\big),
$$
for any $a_{1}, a_{2}\in A$ and $\xi_{1}, \xi_{2}\in A^{\ast}$ (see
\cite[Theorem 3.6]{BLST}).

\begin{pro}\label{pro:fact-doub}
Let $(A, \circ, \ast, \nu, \vartheta)$ be a DPP bialgebra, with the above notations,
$(A\oplus A^{\ast}, \tilde{\circ}, \tilde{\ast}$, $\nu_{\tilde{r}}, \vartheta_{\tilde{r}})$
is a factorizable DPP bialgebra.
\end{pro}

\begin{proof}
We first show that $(A\oplus A^{\ast}, \tilde{\circ}, \tilde{\ast}, \nu_{\tilde{r}},
\vartheta_{\tilde{r}})$ is a DPP bialgebra. Indeed, by \cite[Theorem 2.36]{Bai},
we get $\tilde{r}$ is a solution of $\mathbf{P}_{r}=0$ in $(A\oplus A^{\ast}, \tilde{\circ})$,
and by \cite[Theorem 3.3]{SW}, $\tilde{r}-\tau(\tilde{r})$ is perm-invariant. Moreover,
$\tilde{r}$ is a solution of $\mathbf{L}_{r}=0$ in $(A\oplus A^{\ast}, \tilde{\ast})$ and
$\tilde{r}-\tau(\tilde{r})$ is Leib-invariant (see \cite[Proposition 2.10]{CH}).
Thus, $\tilde{r}$ is a solution of the $\DPYBE$ in $(A\oplus A^{\ast},
\tilde{\circ}, \tilde{\ast})$ and $\tilde{r}-\tau(\tilde{r})$ is DP-invariant.
This means $(A\oplus A^{\ast}, \tilde{\circ}, \tilde{\ast}, \nu_{\tilde{r}},
\vartheta_{\tilde{r}})$ is a quasi-triangular DPP bialgebra.
Finally, since $(A\oplus A^{\ast}, \tilde{\circ}, \nu_{\tilde{r}})$ is a
factorizable perm bialgebra, We get $\mathcal{I}$ is a linear isomorphism.
Thus, $(A\oplus A^{\ast}, \tilde{\circ}, \tilde{\ast}, \nu_{\tilde{r}},
\vartheta_{\tilde{r}})$ is a factorizable DPP bialgebra.
\end{proof}

In \cite{Lin} (resp. \cite{BLST}) the authors have shown that there is a closed relation
between factorizable perm bialgebras (resp. factorizable Leibniz bialgebras) and quadratic
Rota-Baxter perm algebras (resp. quadratic Rota-Baxter Leibniz algebras)
Let $(A, \circ, \ast)$ be a DPP algebra. Recall that a linear map
$R: A\rightarrow A$ is called a {\bf Rota-Baxter operator of
weight $\lambda$} on $(A, \circ, \ast)$ if
\begin{align*}
R(a_{1})\circ R(a_{2})&=R\big(R(a_{1})\circ a_{2}+a_{1}\circ R(a_{2})
+\lambda a_{1}\circ a_{2}\big),\\
R(a_{1})\ast R(a_{2})&=R\big(R(a_{1})\ast a_{2}+a_{1}\ast R(a_{2})
+\lambda a_{1}\ast a_{2}\big),
\end{align*}
for any $a_{1}, a_{2}\in A$. A {\bf Rota-Baxter DPP algebra $(A, \circ, \ast, R)$
of weight $\lambda$} is a DPP algebra $(A, \circ, \ast)$ equipped with a
Rota-Baxter operator $R$ of weight $\lambda$.
For any Rota-Baxter DPP algebra $(A, \circ, \ast, R)$ of weight $\lambda$,
we can obtain a new DPP algebra $(A, \circ_{R}, \ast_{R})$, which is called
the {\bf descendent DPP algebra} of $(A, \circ, \ast, R)$, where
$\circ_{R}, \ast_{R}: A\otimes A\rightarrow A$ are respectively defined by
\begin{align*}
a_{1}\circ_{R}a_{2}&=R(a_{1})\circ a_{2}+a_{1}\circ R(a_{2})+\lambda a_{1}\circ a_{2},\\
a_{1}\ast_{R}a_{2}&=R(a_{1})\ast a_{2}+a_{1}\ast R(a_{2})+\lambda a_{1}\ast a_{2},
\end{align*}
for any $a_{1}, a_{2}\in A$. Furthermore, $R$ is a homomorphism of DPP algebras from
$(A, \circ_{R}, \ast_{R})$ to $(A, \circ, \ast)$. Equipping quadratic DPP algebra
with Rota-Baxter operators satisfying compatibility conditions, we introduce the notion
of quadratic Rota-Baxter DPP algebras.

\begin{defi}\label{def:qua-RB}
The quintuple $(A, \circ, \ast, \omega, R)$ is called a {\bf quadratic Rota-Baxter
DPP algebra of weight $\lambda$} if $(A, \circ, \ast, \omega)$ is a quadratic
DPP algebra, $(A, \circ, \ast, R)$ is a Rota-Baxter DPP algebra of weight
$\lambda$ and satisfying the following compatibility condition: for any $a_{1}, a_{2}\in A$,
\begin{align}
\omega(a_{1}, R(a_{2}))+\omega(R(a_{1}), a_{2})+\lambda\omega(a_{1}, a_{2})=0.\label{qua}
\end{align}
\end{defi}

Let $(A, \circ, \ast)$ be a DPP algebra. It is easy to see that $(A, \circ, \ast,
\omega, R)$ is a quadratic Rota-Baxter DPP algebra of weight $\lambda$ if and
only if $(A, \circ, \omega, R)$ is a quadratic Rota-Baxter perm algebra of weight $\lambda$
and $(A, \ast, \omega, R)$ is a quadratic Rota-Baxter Leibniz algebra of weight $\lambda$.
For the details of quadratic Rota-Baxter perm algebras and quadratic Rota-Baxter Leibniz
algebra see \cite{BLST} and \cite{Lin}. Next, we shoe that there is a close relationship
between the quadratic Rota-Baxter DPP algebras and the quasi-triangular Poisson
bi-dialgebras. For convenience, let us first provide some basic facts about bilinear forms.
Let $A$ be a vector space with a nondegenerate bilinear form $\omega(-,-)$.
The linear isomorphism $J_{\omega}: A^{\ast}\rightarrow A$ is defined by
$$
\langle J_{\omega}^{-1}(a_{1}),\; a_{2}\rangle=\omega(a_{1}, a_{2}),
$$
for any $a_{1}, a_{2}\in A$, which is called the {\bf the linear isomorphism
induced by $\omega(-,-)$}. Moreover, denote by $r_{\omega}\in A\otimes A$ by
$$
\langle r_{\omega},\; \xi_{1}\otimes\xi_{2}\rangle
=\langle J_{\omega}(\xi_{1}),\; \xi_{2}\rangle,
$$
for any $\xi_{1}, \xi_{2}\in A^{\ast}$, which is called the {\bf the 2-tensor form
of $J_{\omega}$}.
Let $(A, \circ, \ast)$ be a DPP algebra and $\omega(-,-)$ be a nondegenerate
bilinear form on $A$. Then $(A, \circ, \omega)$ is a quadratic perm algebra if and
only if $r_{\omega}\in A\otimes A$ defined above is skew-symmetric and perm-invariant
\cite{Lin}, and $(A, \ast, \omega)$ is a quadratic Leibniz algebra if and
only if $r_{\omega}$ is skew-symmetric and Leib-invariant \cite{BLST}.
Thus, we get $(A, \circ, \ast, \omega)$ is a quadratic DPP algebra if and
only if $r_{\omega}$ is skew-symmetric and DP-invariant.

First, we consider the relationship between triangular DPP bialgebras
and quadratic Rota-Baxter DPP algebra of weight 0. Let $(A, \circ, \ast, \omega, R)$
be a quadratic Rota-Baxter DPP algebra of weight 0. Define $r_{\omega}\in A\otimes A$
by $\langle r_{\omega},\; \xi_{1}\otimes\xi_{2}\rangle=\langle R(J_{\omega}(\xi_{1})),\;
\xi_{2}\rangle$ for any $\xi_{1}, \xi_{2}\in A^{\ast}$. By \cite[Proposition 5.5]{Lin},
we get $r_{\omega}$ is a symmetric solution of $\mathbf{L}_{r}=0$, and by
\cite[Corollary 4.18]{BLST}, we get $r_{\omega}$ is a symmetric solution of
$\mathbf{P}_{r}=0$. In this case, $r_{\omega}$ is a symmetric solution of the
$\DPYBE$ in $(A, \circ, \ast)$. Thus, we have

\begin{pro} \label{pro:qua-tri}
Let $(A, \circ, \ast, \omega, R)$ be a quadratic Rota-Baxter DPP algebra of
weight $0$ and and $J_{\omega}: A^{\ast}\rightarrow A$ be the induced linear isomorphism
by $\omega(-,-)$. Then $(A, \circ, \ast, \nu_{r_{\omega}}, \vartheta_{r_{\omega}})$
is a triangular DPP bialgebra, where $r_{\omega}\in A\otimes A$ is given by
$\langle r_{\omega},\; \xi_{1}\otimes\xi_{2}\rangle=\langle R(J_{\omega}(\xi_{1})),\;
\xi_{2}\rangle$ for any $\xi_{1}, \xi_{2}\in A^{\ast}$.
\end{pro}

Second, we show that there is a one-to-one correspondence between factorizable Poisson
bi-dialgebras and quadratic Rota-Baxter DPP algebras of nonzero weight.

\begin{thm}\label{thm:fpb-qrPd}
Let $(A, \circ, \ast, \nu_{r}, \vartheta_{r})$ be a factorizable DPP algebra
with $\mathcal{I}=r^{\sharp}-\tau(r)^{\sharp}$. Then $(A, \circ, \ast, \omega_{\mathcal{I}},
R)$ is a quadratic Rota-Baxter DPP algebra of weight $\lambda$,
where the linear map $R: A\rightarrow A$ and bilinear form $\omega_{\mathcal{I}}(-,-)$
are respectively defined by
$$
R:=\lambda\tau(r)^{\sharp}\mathcal{I}^{-1}\qquad\mbox{and}\qquad
\omega_{\mathcal{I}}(a_{1}, a_{2}):=\langle\mathcal{I}^{-1}(a_{1}),\; a_{2}\rangle,
$$
for any $a_{1}, a_{2}\in A$.

Conversely, for any quadratic Rota-Baxter DPP algebra $(A, \circ, \ast, \omega, R)$
of weight $\lambda$ and a linear isomorphism $J_{\omega}: A^{\ast}\rightarrow A$
induced by $\omega(-,-)$. If $\lambda\neq0$, we define
$$
r^{\sharp}:=\tfrac{1}{\lambda}(R+\lambda\id)
J_{\omega}:\quad A^{\ast}\longrightarrow A,
$$
and define $r\in A\otimes A$ by $\langle r^{\sharp}(\xi),\; \eta\rangle
=\langle r,\; \xi\otimes\eta\rangle$. Then, $r$ is a solution of the Yang-Baxter
equation in $(A, \circ, \ast)$ and induces a factorizable DPP algebra $(A, \circ,
\ast, \nu_{r}, \vartheta_{r})$, where $\nu_{r}$ and $\vartheta_{r}$ are defined by
Eqs. \eqref{cobperm} and \eqref{cobLeb} respectively.
\end{thm}

\begin{proof}
If $(A, \circ, \ast, \nu_{r}, \vartheta_{r})$ be a factorizable DPP bialgebra,
by \cite[Theorem 5.8]{Lin}, we get $(A, \circ, \omega_{\mathcal{I}}, R)$
is a quadratic Rota-Baxter perm algebra of weight $\lambda$,
and by \cite[Theorem 4.22]{BLST}, we get $(A, \ast, \omega_{\mathcal{I}}, R)$
is a quadratic Rota-Baxter Leibniz algebra of weight $\lambda$. Hence,
$(A, \circ, \ast, \omega_{\mathcal{I}}, R)$ is a quadratic Rota-Baxter
DPP algebra of weight $\lambda$.

Conversely, if $(A, \circ, \ast, \omega, R)$ is a quadratic Rota-Baxter DPP algebra
of weight $\lambda$, $\lambda\neq0$, then, by \cite[Proposition 4.13]{Lin}, we obtain
that $r$ is a solution of the $\PYBE$ in $(A, \circ)$ and induces a factorizable
perm bialgebra $(A, \circ, \nu_{r})$, and by \cite[Proposition 4.12]{BLST},
$r$ is a solution of the $\LYBE$ in $(A, \ast)$ and induces a factorizable
Leibniz bialgebra $(A, \cdot, \Delta_{r})$. Thus, $r$ is a solution of the
$\DPYBE$ in $(A, \circ, \ast)$ and $(A, \circ, \ast, \nu_{r},
\vartheta_{r})$ is a factorizable DPP bialgebra.
\end{proof}

Let $(A, \circ, \ast, R)$ be a Rota-Baxter DPP algebra of weight $\lambda$.
Consider the coregular representation $(A^{\ast}, \fl_{A}^{\ast},
\fl_{A}^{\ast}-\fr_{A}^{\ast}, \ffl_{A}^{\ast}, -\ffl_{A}^{\ast}-\ffr_{A}^{\ast})$
of $(A, \circ, \ast)$. We get $(A\ltimes A^{\ast}, \tilde{\circ}, \mathcal{B},
R\oplus\hat{R}^{\ast})$ is a quadratic Rota-Baxter perm algebra of weight $\lambda$ \cite{Lin},
and $(A\ltimes A^{\ast}, \tilde{\ast}, \mathcal{B}, R\oplus\hat{R}^{\ast})$ is a quadratic
Rota-Baxter Leibniz algebra of weight $\lambda$ \cite{BLST}, where $\tilde{\circ}$ and
$\tilde{\ast}$ given by \eqref{sp1} and \eqref{sp2} respectively, the bilinear form
$\mathcal{B}(-,-)$ is given by $\mathcal{B}((a_{1}, \xi_{1}),\; (a_{2}, \xi_{2}))
=\langle a_{1}, \xi_{2}\rangle-\langle a_{2}, \xi_{1}\rangle$ for any $a_{1}, a_{2}\in A$,
$\xi_{1}, \xi_{2}\in A^{\ast}$ and $\hat{R}:=-\lambda\id-R$.
Thus, we get a quadratic Rota-Baxter DPP algebra $(A\ltimes A^{\ast},
\cdot_{\ltimes}, \diamond_{\ltimes}, \mathcal{B}, R\oplus\hat{R}^{\ast})$.
As a direct application of Theorem \ref{thm:fpb-qrPd}, we can provide a
method for constructing factorizable NP bialgebras by the semidirect product.

\begin{cor}\label{cor:rbelam}
Let $(A, \circ, \ast, \omega, R)$ be a quadratic Rota-Baxter DPP algebra of
weight $\lambda$, $\lambda\neq0$, $\{e_{1}, e_{2}, \cdots, e_{n}\}$ be
a basis of $A$ and $\{f_{1}, f_{2}, \cdots, f_{n}\}$ be its dual basis in $A^{\ast}$.
Then $(A\ltimes A^{\ast}, \tilde{\circ}, \tilde{\ast}, \nu_{r}, \vartheta_{r})$
is a factorizable DPP bialgebra, where $\nu_{r}$ and $\vartheta_{r}$ defined
by Eqs. \eqref{cobperm} and \eqref{cobLeb} for
$$
r=\tfrac{1}{\lambda}\sum_{i}\Big(f_{i}\otimes (\lambda\id+R)(e_{i})
+R(e_{i})\otimes f_{i}\Big).
$$
\end{cor}
	
\begin{proof}
Since $(A, \circ, \ast, \omega, R)$ is a quadratic Rota-Baxter DPP algebra of
weight $\lambda$, we get a quadratic Rota-Baxter DPP algebra $(A\ltimes A^{\ast},
\tilde{\circ}, \tilde{\ast}, \mathcal{B}, R\oplus\hat{R}^{\ast})$ of weight $\lambda$.
For any $a\in A$ and $\xi\in A^{\ast}$,
note that $J_{\mathcal{B}}^{-1}: A\ltimes A^{\ast}\longrightarrow(A\ltimes A^{\ast})^{\ast}$,
$(a, \xi)\mapsto(-\xi, a)$, we obtain a linear map $r^{\sharp}:=\tfrac{1}{\lambda}
(R\oplus\hat{R}^{\ast}+\lambda\id)J_{\mathcal{B}}: (A\ltimes A^{\ast})^{\ast}
\rightarrow A\ltimes A^{\ast}$, $\hat{R}=-\lambda\id-R$. That is, for any $a\in A$
and $\xi\in A^{\ast}$,
$$
r^{\sharp}(\xi, a)=\tfrac{1}{\lambda}(R\oplus\hat{R}^{\ast}+\lambda\id)
J_{\mathfrak{B}}(\xi, a)=\tfrac{1}{\lambda}\big(R(a)+\lambda a,\;
-\hat{R}^{\ast}(\xi)-\lambda\xi\big).
$$
Thus, for any $1\leq, i, j\leq n$, we have
\begin{align*}
\langle r,\; e_{i}\otimes f_{j}\rangle
&=\langle r^{\sharp}(e_{i}),\; f_{j}\rangle
=\tfrac{1}{\lambda}\langle R(e_{i})+\lambda e_{i},\; f_{j}\rangle, \\
\langle r,\; f_{j}\otimes e_{i}\rangle
&=\langle r^{\sharp}(f_{j}),\; e_{i}\rangle
=-\tfrac{1}{\lambda}\langle\hat{R}^{\ast}(f_{j})+\lambda f_{j},\; e_{i}\rangle
=-\tfrac{1}{\lambda}\langle f_{j},\; \hat{R}(e_{i})+\lambda e_{i}\rangle
=\tfrac{1}{\lambda}\langle f_{j},\; R(e_{i})\rangle.
\end{align*}
That is, if we define $r\in(A\ltimes A^{\ast})\otimes(A\ltimes A^{\ast})$ by
$\langle r^{\sharp}(\xi_{1}, a_{1}),\; (\xi_{2}, a_{2})\rangle
=\langle r,\; (\xi_{1}, a_{1})\otimes(\xi_{2}, a_{2})\rangle$, then
$r=\tfrac{1}{\lambda}\sum_{i}\big(f_{i}\otimes(\lambda\id+R)(e_{i})
+R(e_{i})\otimes f_{i}\big)$. Thus, by Theorem \ref{thm:fpb-qrPd}, we get
$(A\ltimes A^{\ast}, \tilde{\circ}, \tilde{\ast}, \nu_{r}, \vartheta_{r})$
is a factorizable DPP algebra.
\end{proof}

In particular, we have

\begin{cor}\label{cor:facNPbia}
Let $(A, \circ, \ast)$ be a DPP algebra, $\{e_{1}, e_{2}, \cdots, e_{n}\}$ be
a basis of $A$ and $\{f_{1}, f_{2}, \cdots, f_{n}\}$ be its dual basis in $A^{\ast}$.
Denote $\tilde{r}=\sum_{i}e_{i}\otimes f_{i}$. Then $(A\ltimes A^{\ast}, \tilde{\circ},
\tilde{\ast}, \nu_{\tilde{r}}, \vartheta_{\tilde{r}})$ a factorizable DPP bialgebra.
\end{cor}

\begin{proof}
Since $\id: A\rightarrow A$ is a Rota-Baxter operator of weight $-1$ on the Poisson
dialgebra $(A, \circ, \ast)$, we obtain a quadratic Rota-Baxter DPP algebra
$(A\ltimes A^{\ast}, \tilde{\circ}, \tilde{\ast}, \mathcal{B}, R\oplus0)$ of weight $-1$.
Note that the element $r\in(A\ltimes A^{\ast}) \otimes(A\ltimes A^{\ast})$ in Corollary
\ref{cor:rbelam} is just $\tilde{r}$, we get $(A\ltimes A^{\ast}, \tilde{\circ},
\tilde{\ast}, \nu_{\tilde{r}}, \vartheta_{\tilde{r}})$ a factorizable DPP bialgebra.
\end{proof}

\begin{ex}\label{ex:fact-sum}
Let $(A=\Bbbk\{e_{1}, e_{2}\}, \circ, \ast)$ be the 2-dimensional DPP algebra
considered in Example \ref{ex:alg}, i.e., $e_{2}\circ e_{2}=e_{1}=e_{2}\ast e_{2}$.
Suppose $\{f_{1}, f_{2}\}$ is the dual basis of $\{e_{1}, e_{2}\}$ in $A^{\ast}$.
Then the semidirect product DPP algebra $(A\ltimes A^{\ast}, \tilde{\circ},
\tilde{\ast})$ is given by $A\ltimes A^{\ast}=\Bbbk\{e_{1}, e_{2}, f_{1}, f_{2}\}$,
\begin{align*}
&e_{2}\tilde{\circ} e_{2}=e_{2},\qquad\quad\;
e_{2}\tilde{\circ} f_{2}=-f_{2}=f_{2}\tilde{\circ} e_{2},\\
&e_{2}\tilde{\ast} e_{2}=e_{2}, \qquad\quad
e_{2}\tilde{\ast} f_{2}=-f_{2}=f_{2}\tilde{\ast} e_{2}.
\end{align*}
Then one can check that $\tilde{r}=e_{1}\otimes f_{1}+e_{2}\otimes f_{2}$ is a solution of
the $\DPYBE$ in $(A\ltimes A^{\ast}, \tilde{\circ}, \tilde{\ast})$, and
the induced coproducts are given by
$$
\nu_{\tilde{r}}(e_{2})=-2e_{2}\otimes f_{2},\qquad
\nu_{\tilde{r}}(f_{2})=f_{2}\otimes f_{2}=\vartheta_{\tilde{r}}(f_{2}),\qquad
\vartheta_{\tilde{r}}(e_{2})=e_{2}\otimes f_{2}.
$$
Thus, we obtain a DPP bialgebra $(A\ltimes A^{\ast}, \tilde{\circ},
\tilde{\ast}, \nu_{\tilde{r}}, \vartheta_{\tilde{r}})$. We denote the dual basis of
$\{e_{1}, e_{2}, f_{1}, f_{2}\}$ in $(A\ltimes A^{\ast})^{\ast}$ by $\{e_{1}^{\ast},
e_{2}^{\ast}, f_{1}^{\ast}, f_{2}^{\ast}\}$. Note that the linear maps $\tilde{r}^{\sharp},
\tau(\tilde{r})^{\sharp}: (A\ltimes A^{\ast})^{\ast}\rightarrow A\ltimes A^{\ast}$ are given by
$$
\tilde{r}^{\sharp}(e_{1}^{\ast})=f_{1},\qquad
\tilde{r}^{\sharp}(e_{2}^{\ast})=f_{2},\qquad
\tau(\tilde{r})^{\sharp}(f_{1}^{\ast})=-e_{1},\qquad
\tau(\tilde{r})^{\sharp}(f_{2}^{\ast})=-e_{2},
$$
we get $\mathcal{I}=\tilde{r}^{\sharp}-\tau(\tilde{r})^{\sharp}$ is given by
$$
\mathcal{I}(e_{1}^{\ast})=f_{1},\qquad
\mathcal{I}(e_{2}^{\ast})=f_{2},\qquad
\mathcal{I}(f_{1}^{\ast})=e_{1},\qquad
\mathcal{I}(f_{2}^{\ast})=e_{2}.
$$
Clearly, $\mathcal{I}$ is a linear isomorphism. Thus, $(A\ltimes A^{\ast}, \tilde{\circ},
\tilde{\ast}, \nu_{\tilde{r}}, \vartheta_{\tilde{r}})$ is a
factorizable DPP bialgebra.
\end{ex}

%

\bigskip
\section{Infinite-dimensional DPP bialgebras from Poisson bialgebras} \label{sec:poi-di}
In this section, we recall the notion of a quadratic $\bz$-graded perm algebra, as
a $\bz$-graded perm algebra equipped with an invariant bilinear form. We show that
the tensor product of a finite-dimensional Poisson bialgebra and a quadratic
$\bz$-graded perm algebra can be naturally endowed with a completed DPP bialgebra
structure. Moreover, by discussing the relationship between the completed solution of
the $\DPYBE$ in the induced infinite-dimensional DPP algebra and the solution of the
$\PoiYBE$ in the original Poisson algebra, we get that the induced
infinite-dimensional DPP bialgebra is coboundary (resp. quasi-triangular, triangular)
if the original Poisson bialgebra is coboundary (resp. quasi-triangular, triangular).
First, we construct DPP (co)algebras from Poisson (co)algebras by the tensor product
with $\bz$-graded perm algebras.

\begin{defi}\label{def:zgrad-alg}
A {\bf $\bz$-graded DPP algebra} (resp. {\bf $\bz$-graded perm algebra})
is a DPP algebra $(A, \circ, \ast)$ (resp. a perm algebra $(B, \circ)$) with a
linear decomposition $A=\oplus_{i\in\bz}A_{i}$ (resp. $B=\oplus_{i\in\bz}B_{i}$) such that
each $A_{i}$ (resp. $B_{i}$) is finite-dimensional, $A_{i}\circ A_{j}\subseteq A_{i+j}$
and $A_{i}\ast A_{j}\subseteq A_{i+j}$ (resp. $B_{i}\circ B_{j}\subseteq B_{i+j}$)
for all $i, j\in\bz$.
\end{defi}

\begin{ex}[\cite{LZB}]\label{ex:grperm}
Let $B=\{f_{1}\partial_{1}+f_{2}\partial_{2}\mid f_{1}, f_{2}\in\Bbbk[x_{1}^{\pm},
x_{2}^{\pm}]\}$ and define a binary operation $\circ: B\otimes B\rightarrow B$ by
\begin{align*}
(x_{1}^{i_{1}}x_{2}^{i_{2}}\partial_{s})\circ(x_{1}^{j_{1}}x_{2}^{j_{2}}\partial_{t})
:=\delta_{s,1}x_{1}^{i_{1}+j_{1}+1}x_{2}^{i_{2}+j_{2}}\partial_{t}
+\delta_{s,2}x_{1}^{i_{1}+j_{1}}x_{2}^{i_{2}+j_{2}+1}\partial_{t},
\end{align*}
for any $i_{1}, i_{2}, j_{1}, j_{2}\in\bz$ and $s, t\in\{1, 2\}$.
Then $(B, \circ)$ is a $\bz$-graded perm algebra with the linear decomposition
$B=\oplus_{i\in\bz}B_{i}$, where
$$
B_{i}=\Big\{\sum_{k=1}^{2}f_{k}\partial_{k}\mid f_{k}\text{ is a homogeneous
polynomial with }\deg(f_{k})=i-1,\; k=1,2\Big\}
$$
for all $i\in\bz$.
\end{ex}

By Proposition \ref{pro:poi+perm}, we have the following proposition.

\begin{pro}\label{pro:aff-poi}
Let $(P, \cdot, [-,-])$ be a finite-dimensional Poisson algebra and
$(B=\oplus_{i\in\bz}B_{i}, \circ_{B})$ be a $\bz$-graded perm algebra.
Define a binary operation on $P\otimes B$ by
\begin{align*}
(p_{1}, b_{1})\circ(p_{2}, b_{2})&:=(p_{1}p_{2})\otimes(b_{1}\circ_{B} b_{2}),\\
(p_{1}, b_{1})\ast(p_{2}, b_{2})&:=[p_{1}, p_{2}]\otimes(b_{1}\circ_{B} b_{2}),
\end{align*}
for any $p_{1}, p_{2}\in P$ and $b_{1}, b_{2}\in B$.
Then $(P\otimes B, \circ, \ast)$ is a $\bz$-graded DPP algebra, which is called
the {\bf $\bz$-graded DPP algebra induced by $(P, \cdot, [-,-])$ and $(B, \circ_{B})$}.
\end{pro}

\begin{proof}
By Proposition \ref{pro:poi+perm}, $(P\otimes A, \circ, \ast)$ is a DPP algebra.
Since $(B=\oplus_{i\in\bz}B_{i}, \circ_{B})$ is $\bz$-graded, $(P\otimes B, \circ, \ast)$
is a $\bz$-graded DPP algebra.
\end{proof}

Recall that a {\bf cocommutative coassociative coalgebra} is a vector space $A$ with
a linear map $\Delta: A\rightarrow A\otimes A$ such that $\tau\Delta=\Delta$ and
$(\Delta\otimes\id)\Delta=(\id\otimes\Delta)\Delta$. A {\bf Lie coalgebra} is a vector
space $\g$ with a linear map $\delta: \g\rightarrow\g\otimes\g$ such that $\tau\delta
=-\delta$ and $(\id\otimes\delta)\delta-(\tau\otimes\id)(\id\otimes\delta)\delta
=(\delta\otimes\id)\delta$. If $(P, \Delta)$ is a cocommutative coassociative coalgebra,
$(P, \delta)$ be a Lie coalgebra and $(\id\otimes\Delta)\delta=(\delta\otimes\id)\Delta
+(\tau\otimes\id)(\id\otimes\delta)\Delta$, then $(P, \Delta, \delta)$ is called a
{\bf Poisson coalgebra}. To carry out the Poisson coalgebra affinization, we need to
extend the codomain of the comultiplications $\nu, \vartheta$ to allow infinite sums.
Let $U=\oplus_{i\in\bz}U_{i}$ and $V=\oplus_{j\in\bz}V_{j}$ be $\bz$-graded vector spaces.
We call the {\bf completed tensor product} of $U$ and $V$ to be the vector space
$$
U\,\hat{\otimes}\,V:=\prod_{i,j\in\bz}U_{i}\otimes V_{j}.
$$
If $U$ and $V$ are finite-dimensional, then $U\,\hat{\otimes}\,V$ is just the usual
tensor product $U\otimes V$. In general, an element in $U\,\hat{\otimes}\,V$ is an
infinite formal sum $\sum_{i,j\in\bz}X_{ij} $ with $X_{ij}\in U_{i}\otimes V_{j}$.
So $X_{ij}=\sum_{\alpha} u_{i, \alpha}\otimes v_{j, \alpha}$ for pure tensors
$u_{i, \alpha}\otimes v_{j, \alpha}\in U_{i}\otimes V_{j}$ with $\alpha$ in a finite
index set. Thus a general term of $U\,\hat{\otimes}\,V$ is a possibly infinite sum
$\sum_{i,j,\alpha}u_{i\alpha}\otimes v_{j\alpha}$, where $i, j\in\bz$ and $\alpha$ is
in a finite index set (which might depend on $i, j$). With these notations, for linear
maps $f: U\rightarrow U'$ and $g: V\rightarrow V'$, define
$$
f\,\hat{\otimes}\,g: U\,\hat{\otimes}\,V\rightarrow U'\,\hat{\otimes}\,V',
\qquad \sum_{i,j,\alpha}u_{i,\alpha}\otimes v_{j, \alpha}\mapsto
\sum_{i,j,\alpha} f(u_{i, \alpha})\otimes g(v_{j, \alpha}).
$$
Also the twist map $\tau$ has its completion
$\hat{\tau}: V\,\hat{\otimes}\,V\rightarrow V\,\hat{\otimes}\,V$,
$\sum_{i,j,\alpha}u_{i, \alpha}\otimes v_{j, \alpha}\mapsto
\sum_{i,j,\alpha}v_{j, \alpha}\otimes u_{i, \alpha}$.
Moreover, we define a (completed) comultiplication to be a linear map
$\vartheta: V\rightarrow V\,\hat{\otimes}\,V$,
$\vartheta(v)=\sum_{i, j, \alpha}v_{1, i, \alpha}\otimes v_{2, j, \alpha}$.
Then we have the well-defined map
$$
(\vartheta\,\hat{\otimes}\,\id)\vartheta(v)=(\vartheta\,\hat{\otimes}\,\id)
\Big(\sum_{i,j,\alpha}v_{1, i, \alpha}\otimes v_{2, j, \alpha}\Big)
:=\sum_{i,j,\alpha}\vartheta(v_{1, i, \alpha})\otimes v_{2, j, \alpha}
\in V\,\hat{\otimes}\,V\,\hat{\otimes}\,V.
$$

\begin{defi}\label{def:Pd-coa}
\begin{enumerate}
\item[$(i)$] A {\bf completed perm coalgebra} (resp. {\bf completed Leibniz coalgebra}) is a
    pair $(B, \nu)$ (resp. $(A, \vartheta)$), where $B=\oplus_{i\in\bz}B_{i}$ (resp.
    $A=\oplus_{i\in\bz}A_{i}$) is a $\bz$-graded vector space and $\nu: B\rightarrow
    B\,\hat{\otimes}\,B$ (resp. $\vartheta: A\rightarrow A\,\hat{\otimes}\,A$) is a
    linear map satisfying
\begin{align*}
    &\qquad (\nu\,\hat{\otimes}\,\id)\nu=(\id\,\hat{\otimes}\,\nu)\nu
    =(\tau\,\hat{\otimes}\,\id)(\nu\,\hat{\otimes}\,\id)\nu.\\
    &(\mbox{resp.}\ \
    (\vartheta\,\hat{\otimes}\,\id)\vartheta=(\id\,\hat{\otimes}\,\vartheta)\vartheta
    +(\id\,\hat{\otimes}\,\hat{\tau})(\vartheta\,\hat{\otimes}\,\id)\vartheta.)
\end{align*}
\item[$(ii)$] A {\bf completed Poisson codialgebra} is a triple $(A, \nu, \vartheta)$ where
    $A=\oplus_{i\in\bz}A_{i}$ is a $\bz$-graded vector space and $\nu, \vartheta:
    A\rightarrow A\,\hat{\otimes}\, A$ are linear maps satisfying
\begin{align*}
(\id\,\hat{\otimes}\,\nu)\vartheta&=(\vartheta\,\hat{\otimes}\,\id)\nu
+(\hat{\tau}\,\hat{\otimes}\,\id)(\id\,\hat{\otimes}\,\vartheta)\nu,\\
(\nu\,\hat{\otimes}\,\id)\vartheta&=(\id\,\hat{\otimes}\,\vartheta)\nu
+(\hat{\tau}\,\hat{\otimes}\,\id)(\id\,\hat{\otimes}\,\vartheta)\nu,\\
(\vartheta\,\hat{\otimes}\,\id)\nu&=-(\hat{\tau}\,\hat{\otimes}\,\id)
(\vartheta\,\hat{\otimes}\,\id)\nu.
\end{align*}
\end{enumerate}
\end{defi}

\begin{ex}[\cite{LZB}]\label{ex:grpermco}
Consider the $\bz$-graded vector space $B=\{f_{1}\partial_{1}+f_{2}\partial_{2}\mid
f_{1}, f_{2}\in\Bbbk[x_{1}^{\pm}, x_{2}^{\pm}]\}=\oplus_{i\in\bz}B_{i}$ given in
Example \ref{ex:grperm}. Define a linear map $\nu: B\rightarrow B\,\hat{\otimes}\,B$ by
\begin{align*}
\nu(x_{1}^{m}x_{2}^{n}\partial_{1})&=\sum_{i_{1}, i_{2}\in\bz}
\Big(x_{1}^{i_{1}}x_{2}^{i_{2}}\partial_{1}\otimes x_{1}^{m-i_{1}}
x_{2}^{n-i_{2}+1}\partial_{1} -x_{1}^{i_{1}}x_{2}^{i_{2}}\partial_{2}
\otimes x_{1}^{m-i_{1}+1}x_{2}^{n-i_{2}}\partial_{1}\Big), \\[-2mm]
\nu(x_{1}^{m}x_{2}^{n}\partial_{2})&=\sum_{i_{1}, i_{2}\in\bz}
\Big(x_{1}^{i_{1}}x_{2}^{i_{2}}\partial_{1}\otimes x_{1}^{m-i_{1}}
x_{2}^{n-i_{2}+1}\partial_{2}-x_{1}^{i_{1}}x_{2}^{i_{2}}\partial_{2}
\otimes x_{1}^{m-i_{1}+1}x_{2}^{n-i_{2}}\partial_{2}\Big),
\end{align*}
for any $m, n \in\bz$. Then $(B=\oplus_{i\in\bz}B_{i}, \nu)$ is a completed perm coalgebra.
\end{ex}

Now, we consider the dual version of Proposition \ref{pro:aff-poi} for Poisson coalgebras.

\begin{pro}\label{pro:coperm-copoi}
Let $(P, \Delta, \delta)$ be a finite-dimensional Poisson coalgebra and
$(B=\oplus_{i\in\bz}B_{i}, \nu_{B})$ be a completed perm coalgebra. Define two linear
maps $\nu, \vartheta: P\otimes B\rightarrow(P\otimes B)\,\hat{\otimes}\,(P\otimes B)$ by
\begin{align*}
\nu(p\otimes b)&=\Delta(p)\bullet\nu_{B}(b)=\sum_{(p)}\sum_{i,j,\alpha}
(p_{(1)}\otimes b_{1,i,\alpha})\otimes(p_{(2)}\otimes b_{2,j,\alpha}),\\[-2mm]
\vartheta(p\otimes b)&=\delta(p)\bullet\nu_{B}(b))=\sum_{[p]}\sum_{i,j,\alpha}
(p_{[1]}\otimes b_{1,i,\alpha})\otimes(p_{[2]}\otimes b_{2,j,\alpha}),
\end{align*}
for any $p\in P$ and $b\in B$, where $\Delta(p)=\sum_{(p)}p_{(1)}\otimes p_{(2)}$,
$\delta(p)=\sum_{[p]}p_{[1]}\otimes p_{[2]}$ in the Sweedler notation
and $\nu_{B}(b)=\sum_{i,j,\alpha}b_{1,i,\alpha}\otimes b_{2,j,\alpha}$.
Then $(P\otimes B, \nu, \vartheta)$ is a completed Poisson codialgebra.
\end{pro}

\begin{proof}
For any $\sum_{l}p'_{l}\otimes p''_{l}\otimes p'''_{l}\in P\otimes P\otimes P$ and
$\sum_{i,j,k,\alpha}b'_{i,\alpha}\otimes b''_{j,\alpha}
\otimes b'''_{k,\alpha}\in B\,\hat{\otimes}\,B\,\hat{\otimes}\,B$, we denote
$$
\Big(\sum_{l}p'_{l}\otimes p''_{l}\otimes p'''_{l}\Big)\bullet
\Big(\sum_{i,j,k,\alpha}b'_{i,\alpha}\otimes b''_{j,\alpha}\otimes b'''_{k,\alpha}\Big)
=\sum_{l}\sum_{i,j,k,\alpha}(p'_{l}\otimes b'_{i,\alpha})\otimes
(p''_{l}\otimes b''_{j,\alpha})\otimes(p'''_{l}\otimes b'''_{k,\alpha}).
$$
Then, by using the above notations, since $(B, \nu_{B})$ is a completed perm coalgebra
and $(P, \Delta, \delta)$ is a Poisson coalgebra, for any $p\otimes b\in P\otimes B$, we have
\begin{align*}
(\nu\,\hat{\otimes}\,\id)(\nu(p\otimes b))
&=\big((\Delta\otimes\id)(\Delta(p))\big)\bullet
\big((\nu_{B}\,\hat{\otimes}\,\id)(\nu_{B}(b))\big)\\
&=\big((\id\otimes\Delta)(\Delta(p))\big)\bullet
\big((\id\,\hat{\otimes}\,\nu_{B})(\nu_{B}(b))\big)
=(\id\,\hat{\otimes}\,\nu)(\nu(p\otimes b)).
\end{align*}
Similarly, we have $(\id\,\hat{\otimes}\,\nu)\nu=(\tau\,\hat{\otimes}\,\id)
(\nu\,\hat{\otimes}\,\id)\nu$. Thus, $(P\otimes B, \nu)$ is a completed perm coalgebra.
By \cite[Proposition 3.2]{Hou}, we get $(P\otimes B, \vartheta)$ is a completed Leibniz
coalgebra. Moreover, for any $p\otimes b\in P\otimes B$, note that
\begin{align*}
&\; (\vartheta\,\hat{\otimes}\,\id)(\nu(p\otimes b))
+(\hat{\tau}\,\hat{\otimes}\,\id)((\vartheta\,\hat{\otimes}\,\id)(\nu(p\otimes b)))\\
=&\; \big((\delta\otimes\id)(\Delta(p))\big)\bullet
\big((\nu_{B}\,\hat{\otimes}\,\id)(\nu_{B}(b))\big)
+\big((\tau\otimes\id)(\delta\otimes\id)(\Delta(p))\big)\bullet
\big((\tau\,\hat{\otimes}\,\id)((\nu_{B}\,\hat{\otimes}\,\id)(\nu_{B}(b)))\big)\\
=&\; \big((\delta\otimes\id)(\Delta(p))-(\delta\otimes\id)(\Delta(p))\big)\bullet
\big((\nu_{B}\,\hat{\otimes}\,\id)(\nu_{B}(b))\big)\\
=&\; 0,
\end{align*}
since $(B, \nu_{B})$ is a completed perm coalgebra and $(P, \delta)$ is a Lie coalgebra.
That is, $(\vartheta\,\hat{\otimes}\,\id)\nu=-(\hat{\tau}\,\hat{\otimes}\,\id)
(\vartheta\,\hat{\otimes}\,\id)\nu$. Similarly, $(\id\,\hat{\otimes}\,\nu)\vartheta
=(\vartheta\,\hat{\otimes}\,\id)\nu+(\hat{\tau}\,\hat{\otimes}\,\id)
(\id\,\hat{\otimes}\,\vartheta)\nu$ and $(\nu\,\hat{\otimes}\,\id)\vartheta
=(\id\,\hat{\otimes}\,\vartheta)\nu+(\hat{\tau}\,\hat{\otimes}\,\id)
(\id\,\hat{\otimes}\,\vartheta)\nu$.
Thus, $(P\otimes B, \nu, \vartheta)$ is a completed Poisson codialgebra.
\end{proof}

Next, we consider the DPP bialgebra structure on the tensor product of
a Poisson bialgebra and a $\bz$-graded perm algebra.

\begin{defi}\label{def:quad}
Let $\omega(-, -)$ be a bilinear form on a $\bz$-graded perm algebra
$(B=\oplus_{i\in\bz}B_{i}, \circ)$.

$(i)$ $\omega(-, -)$ is called {\bf invariant}, if $\omega(b_{1}\circ b_{2},\; b_{3})=
        \omega(b_{1},\; b_{2}\circ b_{3}-b_{3}\circ b_{2})$ for any
        $b_{1}, b_{2}, b_{3}\in B$;
$(ii)$ $\omega(-, -)$ is called {\bf graded}, if there exists some $m\in\bz$
        such that $\varpi(B_{i}, B_{j})=0$ when $i+j+m\neq0$.\\
A {\bf quadratic $\bz$-graded perm algebra}, denoted by $(B=\oplus_{i\in\bz}B_{i},
\circ, \omega)$, is a $\bz$-graded perm algebra together with a skew-symmetric
invariant nondegenerate graded bilinear form.

In particular, if $B=B_{0}$, a quadratic $\bz$-graded perm algebra
$(B=\oplus_{i\in\bz}B_{i}, \circ, \omega)$ is just a {\bf quadratic perm algebra}.
\end{defi}

\begin{ex}[\cite{LZB}]\label{ex:qu-perm}
Let $(B=\oplus_{i\in\bz}B_{i}, \circ)$ be the $\bz$-graded perm algebra given in
Example \ref{ex:grperm}, where $B=\{f_{1}\partial_{1}+f_{2}\partial_{2}\mid f_{1},
f_{2}\in\Bbbk[x_{1}^{\pm}, x_{2}^{\pm}]\}=\oplus_{i\in\bz}B_{i}$. Define a skew-symmetric
bilinear form $\omega(-,-)$ on $(B=\oplus_{i\in\bz}B_{i}, \circ)$ by
\begin{align*}
&\omega(x_{1}^{i_{1}}x_{2}^{i_{2}}\partial_{2},\ \ x_{1}^{j_{1}}x_{2}^{j_{2}}\partial_{1})
=-\omega(x_{1}^{j_{1}}x_{2}^{j_{2}}\partial_{1},\ \ x_{1}^{i_{1}}x_{2}^{i_{2}}\partial_{2})
=\delta_{i_{1}+j_{1}, 0}\delta_{i_{2}+j_{2}, 0}, \\
&\qquad\quad\omega(x_{1}^{i_{1}}x_{2}^{i_{2}}\partial_{1},\ \
x_{1}^{j_{1}}x_{2}^{j_{2}}\partial_{1})=\omega(x_{1}^{i_{1}}x_{2}^{i_{2}}\partial_{2},\ \
x_{1}^{j_{1}}x_{2}^{j_{2}}\partial_{2})=0,
\end{align*}
for any $i_{1}, i_{2}, j_{1}, j_{2}\in\bz$.
Then $(B=\oplus_{i\in\bz}B_{i}, \circ, \omega)$ is a quadratic $\bz$-graded perm algebra.
Moreover, $\{x_{1}^{-i_{1}}x_{2}^{-i_{2}}\partial_{2}$,\; $-x_{1}^{-i_{1}}
x_{2}^{-i_{2}}\partial_{1}\mid i_{1}, i_{2}\in\bz\}$ is the dual basis of
$\{x_{1}^{i_{1}}x_{2}^{i_{2}}\partial_{1},\; x_{1}^{i_{1}}x_{2}^{i_{2}}\partial_{2}\mid
i_{1}, i_{2}\in\bz\}$ with respect to $\varpi(-,-)$, consisting of homogeneous elements.
\end{ex}

For a quadratic $\bz$-graded perm algebra $(B=\oplus_{i\in\bz}B_{i}, \circ, \omega)$,
we have $\omega(b_{1}\circ b_{2},\; b_{3})=\omega(b_{2},\; b_{1}\circ b_{3})$
for any $b_{1}, b_{2}, b_{3}\in B$. Moreover, the skew-symmetric nondegenerate
bilinear form $\omega(-,-)$ induces bilinear forms
$$
(\underbrace{B\,\hat{\otimes}\,B\,\hat{\otimes}\,\cdots\,\hat{\otimes}\,
B}_{n\text{-fold}})\otimes(\underbrace{B\otimes B\otimes\cdots
\otimes B}_{n\text{-fold}})\longrightarrow\Bbbk,
$$
for all $n\geq2$, which are denoted by $\hat{\omega}(-,-)$, are defined by
$$
\hat{\omega}\Big(\sum_{i_{1},\cdots,i_{n},\alpha} x_{1, i_{1}, \alpha}
\otimes\cdots\otimes x_{n, i_{n}, \alpha},\ \ y_{1}\otimes\cdots\otimes y_{n}\Big)
=\sum_{i_{1},\cdots,i_{n},\alpha}\prod_{j=1}^{n}
\omega(x_{j, i_{j}, \alpha},\; y_{j}).
$$
Then, one can check that $\hat{\omega}(-,-)$ is {\bf left nondegenerate}, i.e., if
$$
\hat{\omega}\Big(\sum_{i_{1}, \cdots, i_{n},\alpha}x_{1, i_{1}, \alpha}
\otimes\cdots\otimes x_{n, i_{n}, \alpha},\ \ y_{1}\otimes\cdots\otimes y_{n}\Big)
=\hat{\omega}\Big(\sum_{i_{1},\cdots,i_{n},\alpha} z_{1, i_{1}, \alpha}
\otimes\cdots\otimes z_{n, i_{n}, \alpha},\ \ y_{1}\otimes\cdots\otimes y_{n}\Big),
$$
for all homogeneous elements $y_{1}, y_{2},\cdots, y_{n}\in B$, then
$$
\sum_{i_{1},\cdots,i_{n},\alpha} x_{1, i_{1}, \alpha}
\otimes\cdots\otimes x_{n, i_{n}, \alpha}
=\sum_{i_{1},\cdots,i_{n},\alpha} z_{1, i_{1}, \alpha}
\otimes\cdots\otimes z_{n, i_{n}, \alpha}.
$$

By direct calculation, we have

\begin{lem}\label{lem:comp-dual}
Let $(B=\oplus_{i\in\bz}B_{i}, \circ, \omega)$ be a quadratic $\bz$-graded perm algebra.
Define a linear map $\nu_{\omega}: B\rightarrow B\otimes B$ by
$\hat{\omega}(\nu_{\omega}(b_{1}),\; b_{2}\otimes b_{3})
=-\omega(b_{1},\; b_{2}\circ b_{3})$, for any $b_{1}, b_{2}, b_{3}\in B$.
Then $(B, \nu_{\omega})$ is a completed perm coalgebra.
\end{lem}

\begin{ex}[\cite{LZB}]\label{ex:ind-coperm}
Consider the quadratic $\bz$-graded perm algebra $(B=\oplus_{i\in\bz}B_{i}, \circ, \omega)$
given in Example \ref{ex:qu-perm}. Then the induced completed perm coalgebra
$(B=\oplus_{i\in\bz}B_{i}, \nu_{\omega})$ is just the completed perm coalgebra
$(B=\oplus_{i\in\bz}B_{i}, \nu)$ given in Example \ref{ex:grpermco}.
\end{ex}

Recall that a {\bf Lie bialgebra} is a triple $(\g, [-,-], \delta)$ such that
$(\g, [-,-])$ is a Lie algebra, $(\g, \delta)$ is a Lie coalgebra, and the following
compatibility condition holds:
$$
\delta([g_{1}, g_{2}])=(\ad_{\g}(g_{1})\otimes\id+\id\otimes\ad_{\g}(g_{1}))
(\delta(g_{2}))-(\ad_{\g}(g_{2})\otimes\id+\id\otimes\ad_{\g}(g_{2}))(\delta(g_{1})),
$$
where $\ad_{\g}(g_{1})(g_{2})=[g_{1}, g_{2}]$ for all $g_{1}, g_{2}\in\g$.
An {\bf infinitesimal bialgebra} is a triple $(A, \cdot, \Delta)$, where $(A, \cdot)$
is a commutative associative algebra, $(A, \Delta)$ is a cocommutative
coassociative coalgebra and for any $a_{1}, a_{2}\in A$,
$$
\Delta(a_{1}a_{2})=(\fu_{A}(a_{2})\otimes\id)\Delta(a_{1})
+(\id\otimes\,\fu_{A}(a_{1}))\Delta(a_{2}),
$$
where $\fu_{A}(a_{1})(a_{2})=a_{1}a_{2}$ for all $a_{1}, a_{2}\in A$.

\begin{defi}\label{def:Pbialg}
Let $(P, \cdot, [-,-])$ be a Poisson algebra and $(P, \Delta, \delta)$ be a Poisson
coalgebra. Then the quintuple $(P, \cdot, [-,-], \Delta, \delta)$ is called a
{\bf Poisson bialgebra} if
\begin{enumerate}
\item[$(i)$] $(P, \cdot, \Delta)$ is an infinitesimal bialgebra,
\item[$(ii)$] $(P, [-,-], \delta)$ is a Lie bialgebra, 		
\item[$(iii)$] $\Delta$ and $\delta$ are compatible in the following sense:
     for any $p_{1}, p_{2}\in P$,
\begin{align*}
&\Delta([p_{1}, p_{2}])=\big(\ad_{P}(p_{1})\otimes\id+\id\otimes\ad_{P}(p_{1})
\big)\Delta(p_{2})+\big(\fu_{P}(p_{2})\otimes\id
-\id\otimes\fu_{P}(p_{2})\big)\delta(p_{1}),\\
&\qquad\quad \delta(p_{1}p_{2})=(\fu_{P}(p_{1})\otimes\id)\delta(p_{2})
+(\fu_{P}(p_{2})\otimes\id)\delta(p_{1})\\[-1mm]
&\qquad\qquad\qquad\qquad+(\id\otimes\ad_{P}(p_{1}))\Delta(p_{2})
+(\id\otimes\ad_{P}(p_{2}))\Delta(p_{1}).
\end{align*}		
\end{enumerate}
\end{defi}

We now need introduce the concept of completed DPP bialgebras, which is an
infinite-dimensional version of DPP bialgebra. Completed Lie bialgebra was
introduced in \cite{HBG}. Completed Leibniz bialgebra is a generalized completed
Lie bialgebra \cite{Hou}. Recall that a {\bf completed Leibniz bialgebra}
is a triple $(A=\oplus_{i\in\bz}A_{i}, \ast, \vartheta)$ such that $(A, \ast)$ is a
$\bz$-graded Leibniz algebra, $(A, \vartheta)$ is a completed Leibniz coalgebra,
and for any $a_{1}, a_{2}\in A$,
\begin{align*}
&\qquad\qquad\quad \hat{\tau}((\ffr_{A}(a_{2})\,\hat{\otimes}\,\id)(\vartheta(a_{1})))
=(\ffr_{A}(a_{1})\,\hat{\otimes}\,\id)(\vartheta(a_{2})), \\
&\vartheta(a_{1}\ast a_{2})=(\id\,\hat{\otimes}\,\ffr_{A}(a_{2})
-(\ffl_{A}+\ffr_{A})(a_{2})\,\hat{\otimes}\,\id)((\id+\hat{\tau})
(\vartheta(a_{1})))    \\[-1mm]
&\qquad\qquad\qquad\qquad\qquad\qquad\qquad +(\id\,\hat{\otimes}\,\ffl_{A}(a_{1})
+\ffl_{A}(a_{1})\,\hat{\otimes}\,\id)(\vartheta(a_{2})).
\end{align*}
A {\bf completed perm bialgebra} is a triple $(A=\oplus_{i\in\bz}A_{i}, \circ, \nu)$,
where $(A, \circ)$ is a $\bz$-graded perm algebra, $(A, \nu)$ is a completed perm
coalgebra and for any $a_{1}, a_{2}\in A$,
\begin{align*}
&\qquad\qquad (\fr_{A}(a_{1})\,\hat{\otimes}\,\id)\nu(a_{2})
=\hat{\tau}((\fr_{A}(a_{2})\,\hat{\otimes}\,\id)\nu(a_{1})),   \\
&\qquad \nu(a_{1}\circ a_{2})=((\fl_{A}-\fr_{A})(a_{1})\,\hat{\otimes}\,\id)\nu(a_{2})
+(\id\,\hat{\otimes}\,\fr_{A}(a_{2}))\nu(a_{1}),  \\
&\nu(a_{1}\circ a_{2})=(\id\,\hat{\otimes}\,\fl_{A}(a_{1}))\nu(a_{2})
+((\fl_{A}-\fr_{A})(a_{2})\,\hat{\otimes}\,\id)(\nu(a_{1})-\hat{\tau}(\nu(a_{1})).
\end{align*}

\begin{defi}\label{def:CASIbia}
A {\bf completed DPP bialgebra} is a quintuple $(A=\oplus_{i\in\bz}A_{i},
\circ, \ast, \nu, \vartheta)$, where $(A, \circ, \ast)$ is a $\bz$-graded DPP algebra,
$(A, \nu, \vartheta)$ is a completed DPP coalgebra, $(A, \circ, \nu)$ is a
completed perm bialgebra, $(A, \ast, \vartheta)$ is a completed Leibniz bialgebra
and the following compatibility conditions hold:
{\small\begin{align}
&\nu(a_{1}\ast a_{2})=(\id\,\hat{\otimes}\,\ffl_{A}(a_{1})
+\ffl_{A}(a_{1})\,\hat{\otimes}\,\id)(\nu(a_{2}))
+((\fl_{A}-\fr_{A})(a_{2})\,\hat{\otimes}\,\id-\id\,\hat{\otimes}\,\fr_{A}(a_{2}))
((\vartheta+\hat{\tau}(\vartheta))(a_{1})),         \label{cdpbi1} \\
&\qquad \vartheta(a_{1}\circ a_{2})=(\id\,\hat{\otimes}\,\fl_{A}(a_{1}))(\vartheta(a_{2}))
+(\id\,\hat{\otimes}\,\fr_{A}(a_{2}))(\vartheta(a_{1}))
-((\ffl_{A}+\ffr_{A})(a_{1})\,\hat{\otimes}\,\id)(\nu(a_{2}))   \label{cdpbi2} \\[-1mm]
&\qquad\qquad\qquad\qquad-((\ffl_{A}+\ffr_{A})(a_{2})\,\hat{\otimes}\,\id )
((\nu-\hat{\tau}(\nu))(a_{1})),      \nonumber\\
&\qquad\qquad\qquad\qquad(\id\,\hat{\otimes}\,\ffr_{A}(a_{1}))(\hat{\tau}(\nu(a_{2})))
=-(\fr_{A}(a_{2})\,\hat{\otimes}\,\id)(\vartheta(a_{1})),    \label{cdpbi3}\\
&\nu(a_{1}\ast a_{2})=(\id\,\hat{\otimes}\,\fl_{A}(a_{1})
-\fl_{A}(a_{1})\,\hat{\otimes}\,\id)(\vartheta(a_{2}))
+(\id\,\hat{\otimes}\,\ffr_{A}(a_{2})-(\ffl_{A}+\ffr_{A})(a_{2})\,\hat{\otimes}\,\id)
((\nu-\hat{\tau}(\nu))(a_{1})),       \label{cdpbi4} \\
&\qquad(\vartheta+\hat{\tau}(\vartheta))(a_{1}\circ a_{2})
=(\id\,\hat{\otimes}\,\fl_{A}(a_{1}))((\vartheta+\hat{\tau}(\vartheta))(a_{2}))
+(\id\,\hat{\otimes}\,\fl_{A}(a_{2}))((\vartheta+\hat{\tau}(\vartheta))(a_{1}))
\label{cdpbi5} \\[-1mm]
&\qquad\qquad\qquad\qquad\qquad\qquad-(\ffl_{A}(a_{1})\,\hat{\otimes}\,\id)
((\nu-\hat{\tau}(\nu))(a_{2}))-(\ffl_{A}(a_{2})\,\hat{\otimes}\,\id)
((\nu-\hat{\tau}(\nu))(a_{1})),           \nonumber\\
&\qquad\qquad\vartheta(a_{1}\circ a_{2})=(\fl_{A}(a_{1})\,\hat{\otimes}\,\id)(\vartheta(a_{2}))
-(\id\,\hat{\otimes}\,\ffr_{A}(a_{2}))((\nu-\hat{\tau}(\nu))(a_{1}))
\label{cdpbi6} \\[-1mm]
&\qquad\qquad\qquad\qquad+(\id\,\hat{\otimes}\,\ffl_{A}(a_{1}))(\nu(a_{2}))
+((\fl_{A}-\fr_{A})(a_{2})\,\hat{\otimes}\,\id)
((\vartheta+\hat{\tau}(\vartheta))(a_{1})), \nonumber\\
&\qquad\nu(a_{1}\ast a_{2}+a_{2}\ast a_{1})
=(\id\,\hat{\otimes}\,(\ffl_{A}+\ffr_{A})(a_{1}))(\nu(a_{2}))
+((\ffl_{A}+\ffr_{A})(a_{1})\,\hat{\otimes}\,\id)(\hat{\tau}(\nu(a_{2})))
\label{cdpbi7} \\[-1mm]
&\qquad\qquad\qquad\qquad\qquad\qquad+(\id\,\hat{\otimes}\,(\fl_{A}-\fr_{A})(a_{2}))
(\vartheta(a_{1}))+((\fl_{A}-\fr_{A})(a_{2})\,\hat{\otimes}\,\id)(\hat{\tau}
(\vartheta(a_{1}))), \nonumber
\end{align}}
for any $a_{1}, a_{2}\in A$.
\end{defi}

In \cite[Theorem 4.12]{Hou}, we have shown that there is a completed Leibniz bialgebra
structure on the tensor product of a Lie bialgebra and a quadratic
$\bz$-graded perm algebra.

\begin{lem}[\cite{Hou}]\label{lem:Lie+perm=L}
Let $(\g, [-,-], \delta)$ be a finite-dimensional Lie bialgebra, $(B=\oplus_{i\in\bz}B_{i},
\circ_{B}, \omega)$ be a quadratic $\bz$-graded perm algebra and $(\g\otimes B, \ast)$ be
the induced $\bz$-graded Leibniz algebra from $(\g, [-,-])$ by $(B, \circ_{B})$. Define
a linear map $\vartheta: \g\otimes B\rightarrow(\g\otimes B)\otimes(\g\otimes B)$ by
\begin{align}
\vartheta(g\otimes b)=\delta(g)\bullet\nu_{\omega}(b)
:=\sum_{(g)}\sum_{i,j,\alpha}(g_{(1)}\otimes b_{1,i,\alpha})
\otimes(g_{(2)}\otimes b_{2,j,\alpha}), \label{indcoL}
\end{align}
for any $g\in\g$ and $b\in B$, where $\delta(g)=\sum_{(g)}g_{(1)}
\otimes g_{(2)}$ in the Sweedler notation and $\nu_{\omega}(b)=\sum_{i,j,\alpha}
b_{1,i,\alpha}\otimes b_{2,j,\alpha}$. Then $(\g\otimes B, \ast, \vartheta)$ is a
completed Leibniz bialgebra, which is called the {\bf completed Leibniz bialgebra
induced from $(\g, [-,-], \delta)$ by $(B, \circ_{B}, \omega)$}.
\end{lem}

Similarly, we have

\begin{lem}\label{lem:ass+perm=p}
Let $(A, \cdot, \Delta)$ be a finite-dimensional infinitesimal bialgebra,
$(B=\oplus_{i\in\bz}B_{i}, \circ_{B}, \omega)$ be a quadratic $\bz$-graded perm algebra
and $(A\otimes B, \circ)$ be the induced $\bz$-graded perm algebra from $(A, \cdot)$
by $(B, \circ_{B})$. Define a linear map $\nu: A\otimes B\rightarrow(A\otimes B)
\otimes(A\otimes B)$ by
\begin{align}
\nu(a\otimes b)=\Delta(a)\bullet\nu_{\omega}(b)
:=\sum_{(a)}\sum_{i,j,\alpha}(a_{(1)}\otimes b_{1,i,\alpha})
\otimes(a_{(2)}\otimes b_{2,j,\alpha}),  \label{indcoP}
\end{align}
for any $a\in A$ and $b\in B$, where $\Delta(a)=\sum_{(a)}a_{(1)}\otimes a_{(2)}$
in the Sweedler notation and $\nu_{\omega}(b)=\sum_{i,j,\alpha}
b_{1,i,\alpha}\otimes b_{2,j,\alpha}$. Then $(A\otimes B, \circ, \nu)$ is a
completed perm bialgebra, which is called the {\bf completed perm bialgebra
induced from $(A, \cdot, \Delta)$ by $(B, \circ_{B}, \omega)$}.
\end{lem}

\begin{rmk}\label{rmk:ass}
As is well known, the tensor product of a commutative associative algebra and
a Lie (resp. perm, Leibniz, etc.) algebra also has a Lie (resp. perm, Leibniz, etc.)
algebra structure. In fact, the tensor product of infinitesimal bialgebra and many types
of bialgebra structures still maintains this bialgebra structure. More precisely, there
is Lie (resp. perm, Leibniz, etc.) bialgebra structure on the tensor product of
infinitesimal bialgebra and a quadratic Lie (resp. perm, Leibniz, etc.) algebra.
\end{rmk}

Therefore, we can get the first main conclusion of this section.

\begin{thm}\label{thm:P+perm=dP}
Let $(P, \cdot, [-,-], \Delta, \delta)$ be a finite-dimensional Poisson bialgebra,
$(B=\oplus_{i\in\bz}B_{i}, \circ_{B}, \omega)$ be a quadratic $\bz$-graded perm algebra
and $(P\otimes B, \circ, \ast)$ be the induced $\bz$-graded DPP algebra by
$(P, \cdot, [-,-])$ and $(B, \circ_{B})$. Define two linear maps $\nu, \vartheta:
P\otimes B\rightarrow(P\otimes B)\otimes(P\otimes B)$ by Eqs. \eqref{indcoP} and
\eqref{indcoL} respectively. Then $(P\otimes B, \circ, \ast, \nu, \vartheta)$ is a
completed DPP bialgebra, which is called the {\bf completed DPP bialgebra
induced from $(P, \cdot, [-,-], \Delta, \delta)$ by $(B, \circ_{B}, \omega)$}.
\end{thm}

\begin{proof}
By Propositions \ref{pro:aff-poi}, \ref{pro:coperm-copoi} and Lemmas
\ref{lem:Lie+perm=L}, \ref{lem:ass+perm=p}, to prove that $(P\otimes B, \circ, \ast,
\nu, \vartheta)$ is a completed DPP bialgebra, we only need to prove that
Eqs. \eqref{cdpbi1}-\eqref{cdpbi7} hold. In fact, for any $(p_{1}, b_{1})$,
$(p_{2}, b_{2})\in P\otimes B$,
\begin{align*}
&\; \nu((p_{1}, b_{1})\ast(p_{2}, b_{2}))-((\fl_{P\otimes B}-\fr_{P\otimes B})
(p_{2}, b_{2})\,\hat{\otimes}\,\id-\id\,\hat{\otimes}\,\fr_{P\otimes B}(p_{2}, b_{2}))
((\vartheta+\hat{\tau}(\vartheta))(p_{1}, b_{1}))\\[-1mm]
&\qquad\qquad\qquad-(\id\,\hat{\otimes}\,\ffl_{P\otimes B}(p_{1}, b_{1})
+\ffl_{P\otimes B}(p_{1}, b_{1})\,\hat{\otimes}\,\id)(\nu(p_{2}, b_{2}))\\
=&\; \Delta([p_{1}, p_{2}])\bullet\nu_{\omega}(b_{1}\circ_{B}b_{2})
+\Big((\id\otimes\fu_{P}(p_{2}))(\delta(p_{1}))\Big)\bullet
\Big((\id\,\hat{\otimes}\,\fr_{B}(b_{2}))(\nu_{\omega}(b_{1})
-\hat{\tau}(\nu_{\omega}(b_{1})))\Big)\\[-1mm]
&\ \ -\Big((\fu_{P}(p_{2})\otimes\id)(\delta(p_{1}))\Big)\bullet
\Big(((\fl_{B}-\fr_{B})(b_{2})\,\hat{\otimes}\,\id)(\nu_{\omega}(b_{1})
-\hat{\tau}(\nu_{\omega}(b_{1})))\Big)\\[-1mm]
&\ \ -\Big((\id\otimes\ad_{P}(p_{1}))(\Delta(p_{2}))\Big)\bullet
\Big((\id\,\hat{\otimes}\,\fl_{B}(b_{1}))(\nu_{\omega}(b_{2}))\Big)\\[-1mm]
&\ \ -\Big((\ad_{P}(p_{1})\otimes\id)(\Delta(p_{2}))\Big)\bullet
\Big((\fl_{B}(b_{1})\,\hat{\otimes}\,\id)(\nu_{\omega}(b_{2}))\Big).
\end{align*}
Note that, for any $e, f\in B$,
$$
\hat{\omega}(\nu_{\omega}(b_{1}\circ_{B}b_{2}),\ \ e\otimes f)
=-\omega(b_{2},\ \ b_{1}\circ_{B}(e\circ_{B}f))=
\hat{\omega}((\id\,\hat{\otimes}\,\fr_{B}(b_{2}))(\nu_{\omega}(b_{1})),\ \ e\otimes f).
$$
We get $\nu_{\omega}(b_{1}\circ_{B}b_{2})=(\id\,\hat{\otimes}\,\fr_{B}(b_{2}))
(\nu_{\omega}(b_{1}))$ since $\hat{\omega}(-,-)$ is left nondegenerate. Similarly,
we also have $(\fr_{B}(b_{2})\,\hat{\otimes}\,\id)(\nu_{\omega}(b_{1}))=
(\id\,\hat{\otimes}\,\fr_{B}(b_{2}))(\hat{\tau}(\nu_{\omega}(b_{1})))=0$,
$\nu_{\omega}(b_{1}\circ_{B}b_{2})=(\id\,\hat{\otimes}\,\fl_{B}(b_{1}))(\nu_{\omega}(b_{2}))
=(\fl_{B}(b_{1})\,\hat{\otimes}\,\id)(\nu_{\omega}(b_{2}))=
(\fr_{B}(b_{2})\,\hat{\otimes}\,\id)((\hat{\tau}(\nu_{\omega}(b_{1})))-
(\fl_{B}(b_{2})\,\hat{\otimes}\,\id)((\hat{\tau}(\nu_{\omega}(b_{1})))+
(\fl_{B}(b_{2})\,\hat{\otimes}\,\id)(\nu_{\omega}(b_{1}))$.
Hence, we have
\begin{align*}
&\; \nu((p_{1}, b_{1})\ast(p_{2}, b_{2}))-((\fl_{P\otimes B}-\fr_{P\otimes B})
(p_{2}, b_{2})\,\hat{\otimes}\,\id-\id\,\hat{\otimes}\,\fr_{P\otimes B}(p_{2}, b_{2}))
((\vartheta+\hat{\tau}(\vartheta))(p_{1}, b_{1}))\\[-1mm]
&\qquad\qquad\qquad-(\id\,\hat{\otimes}\,\ffl_{P\otimes B}(p_{1}, b_{1})
+\ffl_{P\otimes B}(p_{1}, b_{1})\,\hat{\otimes}\,\id)(\nu(p_{2}, b_{2}))\\
=&\; \Big(\Delta([p_{1}, p_{2}])-\big(\ad_{P}(p_{1})\otimes\id+\id\otimes\ad_{P}(p_{1})
\big)\Delta(p_{2})\\[-2mm]
&\qquad\qquad\qquad\qquad+\big(\fu_{P}(p_{2})\otimes\id-\id\otimes\fu_{P}(p_{2})\big)
\delta(p_{1})\Big)\bullet\nu_{\omega}(b_{1}\circ_{B}b_{2})\\
=&\; 0,
\end{align*}
since $(P, \cdot, [-,-], \Delta, \delta)$ is a Poisson bialgebra. That is,
Eq. \eqref{cdpbi1} holds. Similarly, Eqs. \eqref{cdpbi2}-\eqref{cdpbi7} also hold.
Thus, $(P\otimes B, \circ, \ast, \nu, \vartheta)$ is a completed DPP bialgebra.
\end{proof}

Recall that a Poisson bialgebra $(P, \cdot, [-,-], \Delta, \delta)$ is called {\bf coboundary}
if there exists an element $r\in P\otimes P$ such that $\Delta=\Delta_{r}$ and
$\delta=\delta_{r}$, where
\begin{align}
\Delta_{r}(p)&=(\id\otimes\fu_{P}(p)-\fu_{P}(p)\otimes\id)(r), \label{ass-cobo}\\
\delta_{r}(p)&=(\id\otimes\ad_{P}(p)+\ad_{P}(p)\otimes\id)(r), \label{lie-cobo}
\end{align}
for any $p\in P$. Let $(P, \cdot, [-,-])$ be a Poisson algebra. An element
$r=\sum_{i}x_{i}\otimes y_{i}\in P\otimes P$ is said to be {\bf Poi-invariant}
if $(\id\otimes\fu_{P}(p)-\fu_{P}(p)\otimes\id)(r)=0$ and $(\id\otimes\ad_{P}(p)
+\ad_{P}(p)\otimes\id)(r)=0$ for all $p\in P$. The equations
$$
\mathbf{A}_{r}:=r_{12}r_{13}+r_{13}r_{23}-r_{23}r_{12}=0 \quad\mbox{and}\quad
\mathbf{C}_{r}:=[r_{12}, r_{13}]+[r_{13}, r_{23}]+[r_{12}, r_{23}]=0
$$
is called the (classical) {\bf Poisson Yang-Baxter equation} ($\PoiYBE$) in $(\g, [-,-])$,
where $r_{12}r_{13}=\sum_{i,j}(x_{i}x_{j})\otimes y_{i}\otimes y_{j}$,
$r_{13}r_{23}=\sum_{i,j}x_{i}\otimes x_{j}\otimes(y_{i}y_{j})$,
$r_{23}r_{12}=\sum_{i,j}x_{j}\otimes(x_{i}y_{j})\otimes y_{i}$,
$[r_{12}, r_{13}]=\sum_{i,j}[x_{i}, x_{j}]\otimes y_{i}\otimes y_{j}$,
$[r_{13}, r_{23}]=\sum_{i,j}x_{i}\otimes x_{j}\otimes[y_{i}, y_{j}]$ and
$[r_{12}, r_{23}]=\sum_{i,j}x_{i}\otimes[y_{i}, x_{j}]\otimes y_{j}$.
It is easy to see that $\mathbf{A}_{r}=0$ is just the (classical) associative
Yang-Baxter equation in an associative algebra and $\mathbf{C}_{r}=0$ is just the
classical Yang-Baxter equation in a Lie algebra.

\begin{pro}[\cite{Lin}]\label{pro:poi-bia}
Let $(P, \cdot, [-,-])$ be a Poisson algebra, $r\in P\otimes P$, $\Delta_{r}, \delta_{r}:
P\rightarrow P\otimes P$ be the linear maps defined by Eqs. \eqref{ass-cobo} and \eqref{lie-cobo} respectively.
\begin{enumerate}\itemsep=0pt
\item[$(i)$] If $r$ is a skew-symmetric solution of the $\PoiYBE$ in $(P, \cdot, [-,-])$,
     then $(P, \cdot, [-,-], \Delta_{r}, \delta_{r})$ is a Poisson bialgebra, which is
     called a {\bf triangular Poisson bialgebra} associated with $r$.
\item[$(ii)$] If $r$ is a solution of the $\PoiYBE$ in $(P, \cdot, [-,-])$ and $r+\tau(r)$ is
     Poi-invariant, then $(P, \cdot, [-,-]$, $\Delta_{r}, \delta_{r})$ is a Poisson
     bialgebra, which is called a {\bf quasi-triangular Poisson bialgebra} associated
     with $r$.
\item[$(iii)$] Let $(P, \cdot, [-,-]$, $\Delta_{r}, \delta_{r})$ be quasi-triangular Poisson
     bialgebra associated with $r\in P\otimes P$. If $\mathcal{I}=r^{\sharp}+
     \tau(r)^{\sharp}: P^{\ast}\rightarrow P$ is an isomorphism of vector spaces, then
     $(P, \cdot, [-,-]$, $\Delta_{r}, \delta_{r})$ is called a {\bf factorizable
     Poisson bialgebra}.
\end{enumerate}
\end{pro}

Let $(A=\oplus_{i\in\bz}A_{i}, \circ, \ast)$ be a $\bz$-graded DPP algebra
and $r=\sum_{i,j,\alpha}x_{i\alpha}\otimes y_{j\alpha}\in P\,\hat{\otimes}\,P$.
We denote $r_{12}\circ r_{23}:=\sum_{i,j,k,l,\alpha,\beta}
x_{i,\alpha}\otimes(y_{j,\alpha}\circ x_{k,\beta})\otimes y_{l,\beta}$,
$r_{12}\circ r_{13}:=\sum_{i,j,k,l,\alpha,\beta}(x_{i,\alpha}\circ x_{k,\beta})
\otimes y_{j,\alpha}\otimes y_{l,\beta}$,
$r_{13}\circ r_{23}:=\sum_{i,j,k,l,\alpha,\beta}x_{i,\beta}\otimes x_{k,\alpha}
\otimes(y_{j,\alpha}\circ y_{l,\beta})$
$r_{13}\circ r_{12}:=\sum_{i,j,k,l,\alpha,\beta}(x_{i,\alpha}\circ x_{k,\beta})
\otimes y_{l,\beta}\otimes y_{j,\alpha}$,
$r_{12}\ast r_{13}:=\sum_{i,j,k,l,\alpha,\beta}(x_{i,\alpha}\ast x_{k,\beta})
\otimes y_{j,\alpha}\otimes y_{l,\beta}$,
$r_{12}\ast r_{23}:=\sum_{i,j,k,l,\alpha,\beta}x_{i,\alpha}\otimes
(y_{j,\alpha}\ast x_{k,\beta})\otimes y_{l,\beta}$,
$r_{23}\ast r_{12}:=\sum_{i,j,k,l,\alpha,\beta}x_{k,\alpha}\otimes
(x_{i,\beta}\ast y_{l,\beta})\otimes y_{j,\alpha}$ and
$r_{23}\ast r_{13}:=\sum_{i,j,k,l,\alpha,\beta}x_{k,\alpha}\otimes
x_{i,\beta}\otimes(y_{j,\alpha}\ast y_{l,\beta})$. If $r\in A\,\hat{\otimes}\,A$ satisfies
$\mathbf{P}_{r}:=r_{12}\circ r_{23}-r_{13}\circ r_{23}+r_{12}\circ r_{13}-r_{13}\circ r_{12}=0$
and $\mathbf{L}_{r}=r_{12}\ast r_{13}-r_{12}\ast r_{23}-r_{23}\ast r_{12}+r_{23}\ast r_{13}=0$
in $A\,\hat{\otimes}\,A\,\hat{\otimes}\,A$, then $r$ is called a
{\bf completed solution} of the $\DPYBE$ in $(A=\oplus_{i\in\bz}A_{i}, \circ, \ast)$.
If $A=A_{0}$ is finite-dimensional, a completed solution of the $\DPYBE$ in $(A, \circ, \ast)$
is just a solution of the $\DPYBE$ in DPP algebra $(A=A_{0}, \circ, \ast)$. The same
argument of the proof for \cite[Proposition 3.14, Theorem 3.15]{Lu}
extends to the completed case. We obtain the following proposition.

\begin{pro}\label{pro:dP-tri}
Let $(A=\oplus_{i\in\bz}A_{i}, \circ, \ast)$ be a $\bz$-graded DPP algebra
and $r\in A\,\hat{\otimes}\,A$ is a completed solution of the $\DPYBE$ in
$(A=\oplus_{i\in\bz}A_{i}, \circ, \ast)$. Define a bilinear map
$\vartheta_{r}: A\rightarrow A\,\hat{\otimes}\,A$ by
\begin{align}
\nu_{r}(a):&=\big(\id\,\hat{\otimes}\,\fr_{A}(a)
+(\fr_{A}-\fl_{A})(a)\,\hat{\otimes}\,\id\big)(r), \label{cP-cobo}\\
\vartheta_{r}(a):&=\big((\fl_{A}+\fr_{A})(a)\,\hat{\otimes}\,\id)
-\id\,\hat{\otimes}\,\fr_{A}(a)\big)(r),  \label{cL-cobo}
\end{align}
for any $a\in A$. $(i)$ If $r-\hat{\tau}(r)$ is Poi-invariant, then $(A, \circ, \ast,
\nu_{r}, \vartheta_{r})$ is a completed Poisson di-bialgebra, which is called a {\bf
quasi-triangular completed DPP bialgebra} associated with $r$.
$(ii)$ If $r$ is symmetric, i.e., $r=\hat{\tau}(r)$, then
$(A, \circ, \ast, \nu_{r}, \vartheta_{r})$ is a completed DPP bialgebra,
which is called a {\bf triangular completed DPP bialgebra} associated with $r$.
\end{pro}

These quasi-triangular and triangular bialgebra structures are closely related to the
solutions of the Yang-Baxter equation. To consider their relationships, we first
consider the relationships between the solutions of Yang-Baxter equations.
We now consider the relation between the solutions of the $\PoiYBE$ in a Poisson
algebra and the solutions of the $\DPYBE$ in the induced DPP algebra.
Let $(B=\oplus_{i\in\bz}B_{i}, \circ_{B}, \omega)$ be a quadratic $\bz$-graded perm
algebra and $\{e_{j}\}_{j\in\Omega}$ be a basis of $B=\oplus_{i\in\bz}B_{i}$ consisting
of homogeneous elements. Since $\omega(-,-)$ is graded, skew-symmetric and nondegenerate,
we get a homogeneous dual basis $\{f_{j}\}_{j\in\Omega}$ of $B=\oplus_{i\in\bz}B_{i}$,
which is called the dual basis of $\{e_{i}\}_{i\in\Omega}$ with respect to $\omega(-,-)$,
by $\omega(f_{i}, e_{j})=\delta_{ij}$, where $\delta_{ij}$ is the Kronecker delta.

\begin{pro}\label{pro:PYBE-DPYBE}
Let $(P, \cdot, [-,-])$ be a Poisson algebra, $(B=\oplus_{i\in\bz}B_{i}, \circ_{B},
\omega)$ be a quadratic $\bz$-graded perm algebra, and $(P\otimes B, \circ, \ast)$ be the
induced $\bz$-graded DPP algebra. Suppose that $r=\sum_{i}x_{i}\otimes y_{i}\in
P\otimes P$ is a solution of the $\PoiYBE$ in $(P, \cdot, [-,-])$. Then
\begin{align}
\widehat{r}=\sum_{i}\sum_{j\in\Omega}(x_{i}\otimes e_{j})\otimes(y_{i}\otimes f_{j})
\in(P\otimes B)\,\hat{\otimes}\,(P\otimes B)  \label{r-indP}
\end{align}
is a completed solution of the $\DPYBE$ in $(P\otimes B, \circ, \ast)$,
where $\{e_{j}\}_{j\in\Omega}$ is a homogeneous basis of $B=\oplus_{i\in\bz}B_{i}$
and $\{f_{j}\}_{j\in\Omega}$ is the homogeneous dual basis of $\{e_{j}\}_{j\in\Omega}$
with respect to $\omega(-,-)$. Moreover, we have
\begin{enumerate}\itemsep=0pt
\item[$(i)$] $\widehat{r}$ is symmetric if $r$ is skew-symmetric;
\item[$(ii)$] $\widehat{r}-\hat{\tau}(\widehat{r})$ is DP-invariant if $r+\tau(r)$
     is Poi-invariant.
\end{enumerate}
\end{pro}
	
\begin{proof}
If $r$ is a solution of the $\PoiYBE$ in $(P, \cdot, [-,-])$, i.e.,
$\mathbf{A}_{r}=\mathbf{C}_{r}=0$, we need show $\mathbf{P}_{\widehat{r}}
=\mathbf{L}_{\widehat{r}}=0$. By \cite[Proposition 4.16]{Hou}, we get
$\mathbf{L}_{\widehat{r}}=0$ if $\mathbf{C}_{r}=0$. Here we only show that
$\mathbf{P}_{\widehat{r}}=0$ if $\mathbf{A}_{r}=0$. In fact,
for any $p, q\in\Omega$, we have
\begin{align*}
&\; \widehat{r}_{12}\circ\widehat{r}_{23}-\widehat{r}_{13}\circ\widehat{r}_{23}
+\widehat{r}_{12}\circ\widehat{r}_{13}-\widehat{r}_{13}\circ\widehat{r}_{12}\\
=&\;\sum_{i,j}\sum_{p,q}\Big(\big(x_{i}\otimes(y_{i}x_{j})\otimes y_{j}\big)
\bullet\big(e_{p}\otimes(f_{p}\circ_{B}e_{q})\otimes f_{q}\big)
-\big(x_{i}\otimes x_{j}\otimes(y_{i}y_{j})\big)
\bullet\big(e_{p}\otimes e_{q}\otimes(f_{p}\circ_{B}f_{q})\big)\\[-4mm]
&\qquad\quad+\big((x_{i}x_{j})\otimes y_{i}\otimes y_{j}\big)
\bullet\big((e_{p}\circ_{B}e_{q})\otimes f_{p}\otimes f_{q}\big)
-\big((x_{i}x_{j})\otimes y_{j}\otimes y_{i}\big)\bullet
\big((e_{p}\circ_{B}e_{q})\otimes f_{q}\otimes f_{p}\big)\Big).
\end{align*}
Note that, for given $e_{s}, e_{u}, e_{v}\in B$, $s, u, v\in\Omega$,
\begin{align*}
\hat{\omega}\Big(\sum_{p,q}e_{p}\otimes(f_{p}\circ_{B}e_{q})\otimes f_{q},\ \
e_{s}\otimes e_{u}\otimes e_{v}\Big)
&=\omega(e_{v}\diamond e_{u}-e_{u}\diamond e_{v},\; e_{s}),\\[-2mm]
\hat{\omega}\Big(\sum_{p,q}e_{p}\otimes e_{q}\otimes(f_{p}\circ_{B}f_{q}),\ \
e_{s}\otimes e_{u}\otimes e_{v}\Big)
&=\omega(e_{v}\diamond e_{u}-e_{u}\diamond e_{v},\; e_{s}),\\[-2mm]
\hat{\omega}\Big(\sum_{p,q}(e_{p}\circ_{B}e_{q})\otimes f_{p}\otimes f_{q},\ \
e_{s}\otimes e_{u}\otimes e_{v}\Big)&=\omega(e_{u}\circ_{B}e_{v},\; e_{s}),\\[-2mm]
\hat{\omega}\Big(\sum_{p,q}(e_{p}\circ_{B}e_{q})\otimes f_{q}\otimes f_{p},\ \
e_{s}\otimes e_{u}\otimes e_{v}\Big)&=\omega(e_{v}\circ_{B}e_{u},\; e_{s}).
\end{align*}
If we denote $\Phi_{1}, \Phi_{2}\in B\,\hat{\otimes}\,B\,\hat{\otimes}\,B$ by
$\hat{\varpi}(\Phi_{1},\; e_{s}\otimes e_{u}\otimes e_{v})=\varpi(e_{u}\diamond e_{v},\;
e_{s})$ and $\hat{\varpi}(\Phi_{2},\; e_{s}\otimes e_{u}\otimes e_{v})=\varpi(e_{v}\diamond
e_{u},\; e_{s})$, then we get
$$
\widehat{r}_{12}\ast\widehat{r}_{13}-\widehat{r}_{12}\ast\widehat{r}_{23}
-\widehat{r}_{23}\ast\widehat{r}_{12}+\widehat{r}_{23}\ast\widehat{r}_{13}
=\mathbf{A}_{r}\bullet\Phi_{1}-\mathbf{A}_{r}\bullet\Phi_{2}.
$$
That is, $\mathbf{P}_{\widehat{r}}=0$ if $\mathbf{A}_{r}=0$. Thus,
$\widehat{r}$ is a completed solution of the $\DPYBE$ in $(P\otimes B, \circ, \ast)$
if $r$ is a solution of the $\PoiYBE$ in $(P, \cdot, [-,-])$.

Second, for any $e_{s}, e_{t}\in B$, $s, t\in\Omega$, we have
$$
\hat{\omega}\Big(\sum_{j\in\Omega}e_{j}\otimes f_{j},\; e_{s}\otimes e_{t}\Big)
=\omega(e_{s}, e_{t})=-\varpi(e_{t}, e_{s})
=-\hat{\omega}\Big(\sum_{j\in\Omega}f_{j}\otimes e_{j},\; e_{s}\otimes e_{t}\Big).
$$
That is $\sum_{j}e_{j}\otimes f_{j}=-\sum_{j}f_{j}\otimes e_{j}$. Thus,
we get $\widehat{r}$ symmetric if $r$ is skew-symmetric.

Finally, if $r+\tau(r)$ is Poi-invariant, i.e., $(\id\otimes\fu_{P}(p)
-\fu_{P}(p)\otimes\id)(r+\tau(r))=0$ and $(\id\otimes\ad_{P}(p)+\ad_{P}(p)\otimes\id)
(r+\tau(r))=0$ for all $p\in P$, by \cite[Proposition 4.16]{Hou}, we get
$\widehat{r}-\hat{\tau}(\widehat{r})$ is Leib-invariant. Moreover,
for any $e_{s}, e_{t}\in B$, $s, t\in\Omega$, note that
$$
\hat{\omega}\Big(\sum_{j\in\Omega}(b\circ_{B}e_{j})\otimes f_{j},\ \ e_{s}\otimes e_{t}\Big)
=\omega(b\circ_{B}e_{t},\; e_{s})
=-\hat{\omega}\Big(\sum_{j\in\Omega}(b\circ_{B}f_{j})\otimes e_{j},\ \
e_{s}\otimes e_{t}\Big).
$$
We get $\sum_{j\in\Omega}(b\circ_{B}e_{j})\otimes f_{j}=-\sum_{j\in\Omega}(b\circ_{B}f_{j})
\otimes e_{j}$. Similarly, we also have $\sum_{j\in\Omega}(e_{j}\circ_{B}b)\otimes f_{j}
=-\sum_{j\in\Omega}(f_{j}\circ_{B}b)\otimes e_{j}$ and $\sum_{j\in\Omega}e_{j}\otimes
(f_{j}\circ_{B}b)=-\sum_{j\in\Omega}f_{j}\otimes(e_{j}\circ_{B}b)=\sum_{j\in\Omega}
\big((b\circ_{B}e_{j})\otimes f_{j}-(e_{j}\circ_{B}b)\otimes f_{j}\big)$.
Thus, for any $p\in P$ and $b\in B$,
\begin{align*}
&\;(\id\,\hat{\otimes}\,\fr_{P\otimes B}(p\otimes b)+(\fr_{P\otimes B}
-\fl_{P\otimes B})(p\otimes b)\,\hat{\otimes}\,\id)(\widehat{r}-\hat{\tau}(\widehat{r}))\\
=&\;\sum_{i}\sum_{j\in\Omega}\Big((x_{i}\otimes(y_{i}p))\bullet(e_{j}\otimes(f_{j}\circ_{B}b))
-(y_{i}\otimes(x_{i}p))\bullet(f_{j}\otimes(e_{j}\circ_{B}b))\\[-3mm]
&\qquad\qquad+((x_{i}p)\otimes y_{i})\bullet((e_{j}\circ_{B}b)\otimes f_{j})
-((y_{i}p)\otimes x_{i})\bullet((f_{j}\circ_{B}b)\otimes e_{j})\\[-1mm]
&\qquad\qquad-((px_{i})\otimes y_{i})\bullet((b\circ_{B}e_{j})\otimes f_{j})
+((py_{i})\otimes x_{i})\bullet((b\circ_{B}f_{j})\otimes e_{j})\Big)\\
=&\;\Big((\id\otimes\fu_{P}(p)-\fu_{P}(p)\otimes\id)(r+\tau(r))\Big)
\bullet\Big(\sum_{j\in\Omega}(b\circ_{B}e_{j})\otimes f_{j}-(e_{j}\circ_{B}b)
\otimes f_{j})\Big)\\
=&\; 0.
\end{align*}
Therefore, we obtain $\widehat{r}-\hat{\tau}(\widehat{r})$ is DP-invariant.
The proof is finished.
\end{proof}

Using the conclusion of Proposition \ref{pro:PYBE-DPYBE} and the definition of
quasi-triangular (resp. triangular) DPP bialgebra, we obtain the following
theorem.

\begin{thm}\label{thm:indu-DPbia}
Let $(P, \cdot, [-,-], \Delta, \delta)$ be a finite-dimensional Poisson bialgebra,
$(B=\oplus_{i\in\bz}B_{i}, \circ_{B}, \omega)$ be a quadratic $\bz$-graded perm algebra
and $(P\otimes B, \circ, \ast, \nu, \vartheta)$ be the induced completed Poisson
bi-dialgebra from $(P, \cdot, [-,-], \Delta, \delta)$ by $(B, \circ_{B}, \omega)$. If
$\Delta=\Delta_{r}$ and $\delta=\delta_{r}$ for some $r\in A\otimes A$, then
$\nu=\nu_{\widehat{r}}$ and $\vartheta=\vartheta_{\widehat{r}}$, where $\widehat{r}$
is given by Eq. \eqref{r-indP}.

In particular, we get that $(P\otimes B, \circ, \ast, \nu, \vartheta)$ is a coboundary (resp.
quasi-triangular, triangular) completed DPP bialgebra if
$(P, \cdot, [-,-], \Delta, \delta)$ is coboundary (resp. quasi-triangular, triangular).
\end{thm}

\begin{proof}
Let $r=\sum_{i}x_{i}\otimes y_{i}\in\g\otimes\g$, $\Delta=\Delta_{r}$ and
$\delta=\delta_{r}$. For any $p\in P$ and $b\in B$, we have
$$
\vartheta(p\otimes b)=\sum_{i}\sum_{l,k,\alpha}\Big(
(x_{i}\otimes(py_{i}))\bullet(b_{1,l,\alpha}\otimes b_{2,k,\alpha})
-((px_{i})\otimes y_{i})\bullet(b_{1,l,\alpha}\otimes b_{2,k,\alpha})\Big),
$$
where $\Delta_{r}(p)=(\id\otimes\fu_{P}(p)-\fu_{P}(p)\otimes\id)(r)
=\sum_{i}\big(x_{i}\otimes(py_{i})-(px_{i})\otimes y_{i}\big)$ and
$\nu_{\varpi}(b)=\sum_{l,k,\alpha}b_{1,l,\alpha}\otimes b_{2,k,\alpha}$. On the other hand,
\begin{align*}
\nu_{\widehat{r}}(p\otimes b)&=\big(\id\,\hat{\otimes}\,\fr_{P\otimes B}(p\otimes b)
+(\fr_{P\otimes B}-\fl_{P\otimes B})(p\otimes b)\,\hat{\otimes}\,\id)\big)(\widehat{r})\\
&=\sum_{i}\sum_{j\in\Omega}\Big(\big(x_{i}\otimes(y_{i}p)\big)
\bullet\big(e_{j}\otimes(f_{j}\circ_{B}b)\big)+\big((px_{i})\otimes y_{i}\big)
\bullet\big((e_{j}\circ_{B}b)\otimes f_{j}-(b\circ_{B}e_{j})\otimes f_{j}\big)\Big),
\end{align*}
where $\widehat{r}=\sum_{i}\sum_{j\in\Omega}(x_{i}\otimes e_{j})\otimes(y_{i}\otimes f_{j})$,
$\{e_{j}\}_{j\in\Omega}$ is a homogeneous basis of $B=\oplus_{i\in\bz}B_{i}$
and $\{f_{j}\}_{j\in\Omega}$ is the homogeneous dual basis of $\{e_{j}\}_{j\in\Omega}$
with respect to $\omega(-,-)$.

For two homogeneous basis elements $e_{s}, e_{t}\in B$, $s, t\in\Omega$, since
$$
\hat{\omega}\Big(\sum_{l,k,\alpha}b_{1,l,\alpha}\otimes b_{2,k,\alpha},\;
e_{s}\otimes e_{t}\Big)=\omega(b\circ_{B}e_{t}-e_{t}\circ_{B}b,\; e_{s})
=\hat{\omega}\Big(\sum_{j\in\Omega}e_{j}\otimes(f_{j}\diamond b),\; e_{s}\otimes e_{t}\Big),
$$
we get $\sum_{l,k,\alpha}b_{1,l,\alpha}\otimes b_{2,k,\alpha}=\sum_{j}e_{j}\otimes
(f_{j}\circ_{B}b)$. Similarly, we also have$\sum_{l,k,\alpha}b_{1,l,\alpha}\otimes
b_{2,k,\alpha}=\sum_{j}\big((b\circ_{B}e_{j})\otimes f_{j}-(e_{j}\circ_{B}b)\otimes
f_{j}\big)$. Thus, we obtain $\nu_{\widehat{r}}(p\otimes b)=\nu(p\otimes b)$.
Moreover, by \cite[Theorem 4.17]{Hou}, we get $\vartheta=\vartheta_{\widehat{r}}$
if $\delta=\delta_{r}$. That is to say, $(P\otimes B, \circ, \ast, \nu, \vartheta)
=(P\otimes B, \circ, \ast, \nu_{\widehat{r}}, \vartheta_{\widehat{r}})$ as completed
DPP bialgebras if $\Delta=\Delta_{r}$ and $\delta=\delta_{r}$ for some
$r\in A\otimes A$. Therefore, we get $(P\otimes B, \circ, \ast, \nu, \vartheta)$ is
a coboundary completed DPP bialgebra if $(P, \cdot, [-,-], \Delta, \delta)$
is coboundary. Finally, by Proposition \ref{pro:PYBE-DPYBE}, we get that
$(P\otimes B, \circ, \ast, \nu, \vartheta)$ is a quasi-triangular (resp. triangular)
completed DPP bialgebra if $(P, \cdot, [-,-], \Delta, \delta)$ is
quasi-triangular (resp. triangular).
\end{proof}

It is worth noting that this theorem also gives the following commutative diagram:
$$
\xymatrix@C=3cm@R=0.5cm{
\txt{$r$ \\ {\tiny a solution of the $\PoiYBE$ in $(A, \cdot, [-,-])$}\\
{\tiny such that $r+\tau(r)$ is Poi-invariant}}
\ar[d]_{{\rm Pro.}~\ref{pro:PYBE-DPYBE}}\ar[r]^{{\rm Pro.}~\ref{pro:poi-bia}} &
\txt{$(\g, [-,-], \delta_{r})$ \\ {\tiny a quasi-triangular Poisson bialgebra}}
\ar[d]^{{\rm Thm.}~\ref{thm:indu-DPbia}}_{{\rm Thm.}~\ref{thm:P+perm=dP}} \\
\txt{$\widehat{r}$ \\ {\tiny a completed solution of the $\DPYBE$}\\ {\tiny
in $(\g\otimes B, \ast)$ such that $\widehat{r}-\hat{\tau}(\widehat{r})$ is DP-invariant}}
\ar[r]^{{\rm Pro.}~\ref{pro:dP-tri}\quad} &
\txt{$(P\otimes B, \circ, \ast, \nu, \vartheta)=(P\otimes B, \circ, \ast, \nu_{\widehat{r}},
\vartheta_{\widehat{r}})$ \\
{\tiny a completed quasi-triangular DPP bialgebra}}}
$$

\begin{ex}\label{ex:ind-ctri}
Consider 3-dimensional Poisson algebra $(P=\Bbbk\{e_{1}, e_{2}, e_{3}\}, \cdot, [-,-])$,
where the nonzero product are given by $e_{1}e_{1}=e_{2}$ and $[e_{1}, e_{3}]=e_{3}$.
Let $r=e_{2}\otimes e_{3}-e_{3}\otimes e_{2}$. It is easy to see that $r$ is a
skew-symmetric solution of the $\PoiYBE$ in $(P, \cdot, [-,-])$. Thus we get a
triangular Poisson bialgebra $(P, \cdot, [-,-], \Delta_{r}, \delta_{r})$, where
$\Delta_{r}=0$, $\delta_{r}(e_{1})=e_{2}\otimes e_{3}-e_{3}\otimes e_{2}$ and
$\delta_{r}(e_{2})=\delta_{r}(e_{3})=0$. Consider the quadratic $\bz$-graded perm
algebra $(B=\oplus_{i\in\bz}B_{i}, \circ_{B}, \omega)$ given in Example \ref{ex:qu-perm}.
Then, by Theorem \ref{thm:P+perm=dP}, we obtain a DPP bialgebra $(P\otimes B,
\circ, \ast, \nu, \vartheta)$, where
\begin{align*}
(x_{1}^{i_{1}}x_{2}^{i_{2}}\partial_{s}e_{1})\circ&(x_{1}^{j_{1}}x_{2}^{j_{2}}\partial_{t}e_{1})
=\delta_{s,1}x_{1}^{i_{1}+j_{1}+1}x_{2}^{i_{2}+j_{2}}\partial_{t}e_{1}
+\delta_{s,2}x_{1}^{i_{1}+j_{1}}x_{2}^{i_{2}+j_{2}+1}\partial_{t}e_{1},\\
(x_{1}^{i_{1}}x_{2}^{i_{2}}\partial_{s}e_{1})\ast&(x_{1}^{j_{1}}x_{2}^{j_{2}}\partial_{t}e_{3})
=\delta_{s,1}x_{1}^{i_{1}+j_{1}+1}x_{2}^{i_{2}+j_{2}}\partial_{t}e_{3}
+\delta_{s,2}x_{1}^{i_{1}+j_{1}}x_{2}^{i_{2}+j_{2}+1}\partial_{t}e_{3},\\
(x_{1}^{i_{1}}x_{2}^{i_{2}}\partial_{s}e_{3})\ast&(x_{1}^{j_{1}}x_{2}^{j_{2}}\partial_{t}e_{1})
=-\delta_{s,1}x_{1}^{i_{1}+j_{1}+1}x_{2}^{i_{2}+j_{2}}\partial_{t}e_{3}
-\delta_{s,2}x_{1}^{i_{1}+j_{1}}x_{2}^{i_{2}+j_{2}+1}\partial_{t}e_{3},\\
\vartheta(x_{1}^{m}x_{2}^{n}\partial_{1}e_{1})&=\sum_{i_{1}, i_{2}\in\bz}
\Big(x_{1}^{i_{1}}x_{2}^{i_{2}}\partial_{1}e_{2}\otimes x_{1}^{m-i_{1}}
x_{2}^{n-i_{2}+1}\partial_{1}e_{3}-x_{1}^{i_{1}}x_{2}^{i_{2}}\partial_{2}e_{2}
\otimes x_{1}^{m-i_{1}+1}x_{2}^{n-i_{2}}\partial_{1}e_{3}\\[-4mm]
&\qquad\quad+x_{1}^{i_{1}}x_{2}^{i_{2}}\partial_{2}e_{3}
\otimes x_{1}^{m-i_{1}+1}x_{2}^{n-i_{2}}\partial_{1}e_{2}
-x_{1}^{i_{1}}x_{2}^{i_{2}}\partial_{1}e_{3}\otimes x_{1}^{m-i_{1}}
x_{2}^{n-i_{2}+1}\partial_{1}e_{2}\Big), \\
\vartheta(x_{1}^{m}x_{2}^{n}\partial_{2}e_{1})&=\sum_{i_{1}, i_{2}\in\bz}
\Big(x_{1}^{i_{1}}x_{2}^{i_{2}}\partial_{1}e_{2}\otimes x_{1}^{m-i_{1}}
x_{2}^{n-i_{2}+1}\partial_{2}e_{3}-x_{1}^{i_{1}}x_{2}^{i_{2}}\partial_{2}e_{2}
\otimes x_{1}^{m-i_{1}+1}x_{2}^{n-i_{2}}\partial_{2}e_{3}\\[-4mm]
&\qquad\quad+x_{1}^{i_{1}}x_{2}^{i_{2}}\partial_{2}e_{3}
\otimes x_{1}^{m-i_{1}+1}x_{2}^{n-i_{2}}\partial_{2}e_{2}
-x_{1}^{i_{1}}x_{2}^{i_{2}}\partial_{1}e_{3}\otimes x_{1}^{m-i_{1}}
x_{2}^{n-i_{2}+1}\partial_{2}e_{2}\Big),
\end{align*}
and others all are zero.

On the other hand, by the skew-symmetric solution $r$ of $\PoiYBE$ in $(P, \cdot, [-,-])$,
we can get obtain a symmetric solution of $\DPYBE$ in $(P\otimes B, \circ, \ast)$:
\begin{align*}
\widehat{r}&=\sum_{i_{1}, i_{2}\in\bz}\Big(x_{1}^{i_{1}}x_{2}^{i_{2}}\partial_{1}e_{2}
\otimes x_{1}^{-i_{1}}x_{2}^{-i_{2}}\partial_{2}e_{3}
-x_{1}^{i_{1}}x_{2}^{i_{2}}\partial_{1}e_{3}\otimes
x_{1}^{-i_{1}}x_{2}^{-i_{2}}\partial_{2}e_{2}\\[-5mm]
&\qquad\qquad -x_{1}^{i_{1}}x_{2}^{i_{2}}\partial_{2}e_{2}
\otimes x_{1}^{-i_{1}}x_{2}^{-i_{2}}\partial_{1}e_{3}
+x_{1}^{i_{1}}x_{2}^{i_{2}}\partial_{2}e_{3}
\otimes x_{1}^{-i_{1}}x_{2}^{-i_{2}}\partial_{1}e_{2}\Big).
\end{align*}
This symmetric solution induces a triangular DPP bialgebra $(P\otimes B,
\circ, \ast, \nu_{\widehat{r}}, \vartheta_{\widehat{r}})$. One can check
$\nu_{\widehat{r}}=\nu$ and $\vartheta_{\widehat{r}}=\vartheta$ as defined above.
\end{ex}

At the end of this section, we consider the finite-dimensional DPP bialgebras
constructed from Poisson bialgebras. First, in Theorem \ref{thm:indu-DPbia},
if the quadratic $\bz$-graded perm algebra $(B=\oplus_{i\in\bz}B_{i}, \circ_{B}, \omega)$
is a finite-dimensional quadratic perm algebra, then we get that the induced
finite-dimensional DPP bialgebra $(P\otimes B, \circ, \ast, \nu, \vartheta)$
is coboundary (resp. quasi-triangular, triangular) if $(P, \cdot, [-,-], \Delta, \delta)$
is coboundary (resp. quasi-triangular, triangular). Following, we consider the factorizable
DPP bialgebra obtained by this method.

\begin{pro}\label{pro:indu-factbia}
Let $(P, \cdot, [-,-], \Delta, \delta)$ be a finite-dimensional Poisson bialgebra,
$(B, \circ_{B}, \omega)$ be a quadratic perm algebra and $(P\otimes B, \circ, \ast,
\nu, \vartheta)$ be the induced DPP bialgebra from $(P, \cdot, [-,-], \Delta,
\delta)$ by $(B, \circ_{B}, \omega)$. Then we have
$(P\otimes B, \circ, \ast, \nu, \vartheta)$ is factorizable if $(P, \cdot, [-,-],
\Delta, \delta)$ is factorizable.
\end{pro}

\begin{proof}
Suppose $(P, \cdot, [-,-], \Delta, \delta)$ is a factorizable Poisson bialgebra,
i.e., $\Delta=\Delta_{r}$ and $\delta=\delta_{r}$ for some $r=\sum_{i}x_{i}\otimes
y_{i}\in P\otimes P$ and $\mathcal{I}=r^{\sharp}+\tau(r)^{\sharp}: P^{\ast}\rightarrow P$
is an isomorphism of vector spaces. By Theorem \ref{thm:indu-DPbia}, we get that
$(P\otimes B, \circ, \ast, \nu, \vartheta)=(P\otimes B, \circ, \ast, \nu_{\widehat{r}},
\vartheta_{\widehat{r}})$ is a DPP bialgebra, where $\widehat{r}$ is given
by Eq. \eqref{r-indP}. Here we need to show that $\widehat{\mathcal{I}}:=
\widehat{r}^{\sharp}-\tau(\widehat{r})^{\sharp}: (P\otimes B)^{\ast}\rightarrow
P\otimes B$ is an isomorphism of vector spaces. Denote $\kappa:=\sum_{j}e_{j}\otimes
f_{j}\in B\otimes B$, where $\{e_{1}, e_{2},\cdots, e_{n}\}$ is a basis of $B$
and $\{f_{1}, f_{2},\cdots, f_{n}\}$ is the dual basis of it with respect to $\omega(-,-)$.
Define $\kappa^{\sharp}: B^{\ast}\rightarrow B$ by
$\langle\kappa^{\sharp}(\xi_{1}),\; \xi_{2}\rangle=\langle\xi_{1}
\otimes\xi_{2},\; \kappa\rangle$, for any $\xi_{1}, \xi_{2}\in B^{\ast}$.
Then, one can check that $\kappa^{\sharp}$ is a linear isomorphism and
$\langle\kappa^{\sharp}(\xi_{1}),\; \xi_{2}\rangle=\langle\xi_{1}\otimes\xi_{2},\;
\kappa\rangle=-\langle\xi_{2}\otimes\xi_{1},\; \kappa\rangle
=-\langle\kappa^{\sharp}(\xi_{2}),\; \xi_{1}\rangle$.
Therefore, we have
\begin{align*}
\langle\widehat{r}^{\sharp}(\eta_{1}\otimes\xi_{1}),\;\eta_{2}\otimes\xi_{2}\rangle
&=\sum_{i,j}\langle(\eta_{1}\otimes\xi_{1})\otimes(\eta_{2}\otimes\xi_{2}),\ \
(x_{i}\otimes e_{j})\otimes(y_{i}\otimes f_{j})\rangle\\[-2mm]
&=\Big(\sum_{j}\langle\eta_{1}, x_{i}\rangle\langle\eta_{2}, y_{i}\rangle\Big)
\Big(\sum_{i}\langle\xi_{1}, e_{j}\rangle\langle\xi_{2}, f_{j}\rangle\Big)\\[-2mm]
&=\langle r^{\sharp}(\eta_{1}),\; \eta_{2}\rangle
\langle\kappa^{\sharp}(\xi_{1}),\; \xi_{2}\rangle\\
&=\langle r^{\sharp}(\eta_{1})\otimes\kappa^{\sharp}(\xi_{1}),\;\eta_{2}\otimes\xi_{2}\rangle.
\end{align*}
for any $\xi_{1}, \xi_{2}\in B^{\ast}$ and $\eta_{1}, \eta_{2}\in P^{\ast}$.
That is, $\widehat{r}^{\sharp}=r^{\sharp}\otimes\kappa^{\sharp}$, Similarly,
$\tau(\widehat{r})^{\sharp}=-\tau(r)^{\sharp}\otimes\kappa^{\sharp}$. Thus,
$\widehat{\mathcal{I}}=\mathcal{I}\otimes\kappa^{\sharp}$ is an isomorphism of
vector spaces, $(P\otimes B, \circ, \ast, \nu, \vartheta)$ is a
factorizable DPP bialgebra.
\end{proof}

In \cite{Kup}, Kupershmidt found that the $\CYBE$ in tensor form on Lie algebras can
be converted into an $\mathcal{O}$-operator associated to the coadjoint representation.
This conclusion has been confirmed on various types of algebras.
Let $(A, \circ, \ast)$ be a DPP algebra and $(V, \kl, \kr, \kkl, \kkr)$ be a
representation of it. A linear map $T: V\rightarrow A$ is called an
{\bf $\mathcal{O}$-operator of $(A, \circ, \ast)$ associated to $(V, \kl, \kr,
\kkl, \kkr)$} if for any $v_{1}, v_{2}\in V$,
\begin{align*}
T(v_{1})\circ T(v_{2})&=T\big(\kl(T(v_{1}))(v_{2})+\kr(T(v_{2}))(v_{1})\big),\\
T(v_{1})\ast T(v_{2})&=T\big(\kkl(T(v_{1}))(v_{2})+\kkr(T(v_{2}))(v_{1})\big).
\end{align*}

\begin{pro}[\cite{Lu}]\label{pro:o-dp}
Let $(A, \circ, \ast)$ be a DPP algebra and $r\in A\otimes A$ be symmetric.
Then $r$ is a solution of the $\DPYBE$ in $(A, \circ, \ast)$ if and only if
$r^{\sharp}: A^{\ast}\rightarrow A$ is an $\mathcal{O}$-operator of $(A, \circ, \ast)$
associated to the coregular representation $(A^{\ast}, -\fl_{A}^{\ast},
\fr_{A}^{\ast}-\fl_{A}^{\ast}, \ffl_{A}^{\ast}, -\ffl_{A}^{\ast}-\ffr_{A}^{\ast})$.
\end{pro}

The $\mathcal{O}$-operator of Poisson algebra was considered in \cite{NB}.
Let $(P, \cdot, [-,-])$ be a Poisson algebra, $V$ be a vector space and
$\mu, \rho: P\rightarrow\gl(V)$ be two linear maps. Recall that $(V, \mu, \rho)$ is
called a {\bf representation of} $(P, \cdot, [-,-])$, if $(V, \rho)$ is a
representation of the Lie algebra $(P, [-,-])$ and $(V, \mu)$ is a representation
of the commutative associative algebra $(P, \cdot)$ and $\mu, \rho$ satisfy
the following conditions: $\rho(p_{1}p_{2})=\mu(p_{2})\rho(p_{1})+\mu(p_{1})
\rho(p_{2})$ and $\mu([p_{1}, p_{2}])=\rho(p_{1})\mu(p_{2})-\mu(p_{2})\rho(p_{1})$
for any $p_{1}, p_{2}\in P$. In particular, $(P, \fu_{P}, \ad_{P})$ is a
representation of the Poisson algebra $(P, \cdot, [-,-])$, which is called the
{\bf regular representation} of $(P, \cdot, [-,-])$. Moreover, $(P^{\ast}, -\fu^{\ast}_{P},
\ad^{\ast}_{P})$ is also a representation of the Poisson algebra $(P, \cdot, [-,-])$,
which is called the {\bf coregular representation} of $(P, \cdot, [-,-])$.
Recall that an {\bf $\mathcal{O}$-operator of a Poisson algebra $(P, \cdot, [-,-])$
associated to a representation $(V, \mu, \rho)$} is a linear map $T: V\rightarrow P$ such that
\begin{align*}
T(v_{1})T(v_{2})&=T\big(\mu(T(v_{1}))(v_{2})+\mu(T(v_{2}))(v_{1})\big),\\
[T(v_{1}), T(v_{2})]&=T\big(\rho(T(v_{1}))(v_{2})-\rho(T(v_{2}))(v_{1})\big),
\end{align*}
for any $v_{1}, v_{2}\in V$.

\begin{pro}[\cite{NB}]\label{pro:o-poi}
Let $(P, \cdot, [-,-])$ be a Poisson algebra and $r\in P\otimes P$ be skew-symmetric.
Then $r$ is a solution of the $\PoiYBE$ in $(P, \cdot, [-,-])$ if and only if $r^{\sharp}$
is an $\mathcal{O}$-operator of $(P, \cdot, [-,-])$ associated to the coregular
representation $(P^{\ast}, -\fu_{P}^{\ast}, \ad_{P}^{\ast})$.
\end{pro}

Let $(P, \cdot, [-,-])$ be a Poisson algebra, $r\in P\otimes P$ and $(B, \circ_{B},
\omega)$ be a quadratic perm algebra. By the proofs of Theorem \ref{thm:indu-DPbia}
and Proposition \ref{pro:indu-factbia}, we also have the following commutative diagram:
$$
\xymatrix@C=3cm@R=0.5cm{
\txt{$r$ \\ {\tiny a skew-symmetric solution} \\ {\tiny of the $\PoiYBE$ in
$(P, \cdot, [-,-])$}} \ar[d]_-{{\rm Pro.}~\ref{pro:PYBE-DPYBE}}\ar[r]^-{{\rm Pro.}~\ref{pro:o-poi}} &
\txt{$r^{\sharp}$\\ {\tiny an $\mathcal{O}$-operator of $(P, \cdot, [-,-])$} \\
{\tiny associated to $(P^{\ast}, -\fu_{P}^{\ast}, \ad_{P}^{\ast})$}}
\ar[d]^-{\mbox{$-\otimes\kappa^{\sharp}$}} \\
\txt{$\widetilde{r}$ \\ {\tiny a symmetric solution} \\ {\tiny of the $\DPYBE$ in
$(P\otimes B, \circ, \ast)$}} \ar[r]^-{{\rm Pro.}~\ref{pro:o-dp}}
& \txt{$\widehat{r}^{\sharp}=r^{\sharp}\otimes\kappa^{\sharp}$ \\
{\tiny an $\mathcal{O}$-operator of $(P\otimes B, \circ, \ast)$ associated to } \\
{\tiny $((P\otimes B)^{\ast}, -\fl_{\g\otimes B}^{\ast},
\fr_{\g\otimes B}^{\ast}-\fl_{\g\otimes B}^{\ast}, \ffl_{\g\otimes B}^{\ast},
-\ffl_{\g\otimes B}^{\ast}-\ffr_{\g\otimes B}^{\ast})$}}}
$$

Finally, we give a simple example.

\begin{ex}\label{ex:ind-tri}
Consider 3-dimensional triangular Poisson bialgebra $(P=\Bbbk\{e_{1}, e_{2}, e_{3}\},
\cdot, [-,-]$, $\Delta_{r}, \delta_{r})$ given in Example \ref{ex:ind-ctri}, i.e.,
$r=e_{2}\otimes e_{3}-e_{3}\otimes e_{2}$, $e_{1}e_{1}=e_{2}$, $[e_{1}, e_{3}]=e_{3}$,
$\Delta_{r}=0$, $\delta_{r}(e_{1})=e_{2}\otimes e_{3}-e_{3}\otimes e_{2}$ and
$\delta_{r}(e_{2})=\delta_{r}(e_{3})=0$. Let $2$-dimensional $B=\Bbbk\{x_{1}, x_{2}\}$.
If we define $x_{1}\circ_{B}x_{1}=x_{1}$ and $x_{1}\circ_{B} x_{2}=x_{2}$, then
$(B, \circ_{B})$ is a perm algebra. Moreover, if we define a skew-symmetric invariant
nondegenerate bilinear form $\omega(-,-)$ on $(B, \circ_{B})$ by $\omega(x_{1}, x_{2})=1
=-\omega(x_{2}, x_{1})$, then $(B, \circ, \omega)$ is a quadratic perm algebra.
By Theorem \ref{thm:P+perm=dP}, we obtain a DPP bialgebra $(P\otimes B,
\circ, \ast, \nu, \vartheta)$, where
\begin{align*}
&\qquad\quad \vartheta(e_{1}\otimes x_{1})=(e_{2}\otimes x_{2})\otimes(e_{3}\otimes x_{1})
-(e_{3}\otimes x_{2})\otimes(e_{2}\otimes x_{1}),\\
&\qquad\quad \vartheta(e_{1}\otimes x_{1})=(e_{2}\otimes x_{2})\otimes(e_{3}\otimes x_{2})
-(e_{3}\otimes x_{2})\otimes(e_{2}\otimes x_{2}),\\
& (e_{1}\otimes x_{1})\circ(e_{1}\otimes x_{1})=e_{1}\otimes x_{1},\qquad
\qquad (e_{1}\otimes x_{1})\circ(e_{1}\otimes x_{2})=e_{1}\otimes x_{2},\\
& (e_{1}\otimes x_{1})\ast(e_{3}\otimes x_{1})=e_{3}\otimes x_{1},\qquad
\qquad (e_{3}\otimes x_{1})\ast(e_{1}\otimes x_{2})=-e_{3}\otimes x_{2},\\
& (e_{1}\otimes x_{1})\ast(e_{3}\otimes x_{2})=e_{3}\otimes x_{2},\qquad
\qquad (e_{3}\otimes x_{1})\ast(e_{1}\otimes x_{2})=-e_{3}\otimes x_{2},
\end{align*}
and others all are zero. Define
$$
\widehat{r}=(e_{2}\otimes x_{2})\ast(e_{3}\otimes x_{1})
+(e_{3}\otimes x_{1})\ast(e_{2}\otimes x_{2})-(e_{2}\otimes x_{1})\ast(e_{3}\otimes x_{2})
-(e_{3}\otimes x_{2})\ast(e_{2}\otimes x_{1}).
$$
Then $\widehat{r}$ is a summetric solution of the $\DPYBE$ in $(P\otimes B, \circ, \ast)$.
By direct calculation, we get the linear map $r^{\sharp}: P^{\ast}\rightarrow P$ is given
by 
$$
r^{\sharp}(\xi_{1})=0,\qquad\qquad r^{\sharp}(\xi_{2})=e_{3},\qquad\qquad r^{\sharp}(\xi_{3})=-e_{2},
$$
where $\xi_{1}, \xi_{2}, \xi_{3}\in P^{\ast}$ is the dual basis of $\{e_{1}, e_{2}, e_{3}\}$;
the linear map $\kappa^{\sharp}: B^{\ast}\rightarrow B$ is given by 
$$
r^{\sharp}(\eta_{1})=-x_{2},\qquad\qquad\qquad r^{\sharp}(\eta_{2})=x_{1},
$$ 
where $\kappa=x_{2}\otimes x_{1}-x_{1}\otimes x_{2}$, 
$\eta_{1}, \eta_{2}\in B^{\ast}$ is the dual basis of $\{x_{1}, x_{2}\}$; 
the linear map $\widehat{r}^{\sharp}:(P\otimes B)^{\ast}\rightarrow
P\otimes B$ is given by 
\begin{align*}
&\widehat{r}^{\sharp}(\xi_{2}\otimes\eta_{1})=-e_{3}\otimes x_{2},\qquad\qquad
\widehat{r}^{\sharp}(\xi_{2}\otimes\eta_{2})=e_{3}\otimes x_{1},\\
&\widehat{r}^{\sharp}(\xi_{3}\otimes\eta_{1})=e_{2}\otimes x_{2},\qquad\qquad\ \ \;
\widehat{r}^{\sharp}(\xi_{3}\otimes\eta_{2})=-e_{2}\otimes x_{1}.
\end{align*}
It is easy to see $\widehat{r}^{\sharp}=r^{\sharp}\otimes\kappa^{\sharp}$, and it is
an $\mathcal{O}$-operator of $(P\otimes B, \circ, \ast)$ associated to the coregular
representation.
\end{ex}

\bigskip
\noindent
{\bf Acknowledgements. } This work was financially supported by
National Natural Science Foundation of China (No. 11771122).

\smallskip
\noindent
{\bf Declaration of interests.} The authors have no conflicts of interest to disclose.

\smallskip
\noindent
{\bf Data availability.} Data sharing is not applicable to this article as no new data were
created or analyzed in this study.

 \end{document}